 \documentclass[10pt,reqno]{amsart}
  \usepackage{geometry}
\usepackage{amsfonts,amsthm}
\usepackage{enumerate}
\theoremstyle{plain}
\newtheorem*{thm A}{Theorem~A}
\newtheorem*{thm B}{Theorem~B}
\newtheorem*{thm C}{Theorem~C}
\newtheorem*{thm D}{Theorem~D}
\newtheorem*{thm E}{Theorem~E}
\newtheorem*{thm F}{Theorem~F}
\newtheorem*{thm 1}{Theorem~1}
\newtheorem*{thm 2}{Theorem~2}
\newtheorem*{thm 3}{Theorem~3}
\newtheorem*{pro A}{Proposition~A}
\newtheorem*{pro B}{Proposition~B}
\newtheorem*{pro C}{Proposition~C}
\newtheorem*{lem A}{Lemma~A}
\newtheorem*{lem B}{Lemma~B}
\newtheorem{theorem}{Theorem}[section]
\newtheorem{lemma}[theorem]{Lemma}

\newtheorem{remark}[theorem]{Remark}

\theoremstyle{plain}

\def \GBt{G_2({\mathbb C}^{m+2})}
\def \GBo{G_2({\mathbb C}^{m+1})}

\def \R{R_{\xi}}
\def \al{\alpha}
\def \N{\nabla}
\def \xn{{\xi}_{\nu}}
\def \xo{{\xi}_1}
\def \xt{{\xi}_2}
\def \xh{{\xi}_3}
\def \x{\xi}
\def \p{\phi}
\def \pn{\phi_{\nu}}
\def \po{\phi_{1}}

\def \e{\eta}
\def \en{\eta_{\nu}}
\def \QP{{\mathcal Q}^{\bot}}
\def \Q{\mathcal Q}
\def \E{\eta}
\def \EN{{\eta}_{\nu}}
\def \KN{{\xi}_{\nu}}
\def \Ph{{\phi}}
\def \PN{{\phi}_{\nu}}
\def \X{X_{0}}
\def \SN{\sum_{\nu=1}^3}

\begin{document}

\title[quadratic Killing structure Jacobi operator]{Quadratic Killing structure Jacobi operator for real hypersurfaces in complex two-plane Grassmannians}
  \vspace{0.2in}

\author[H. Lee, Y.J. Suh, \& C. Woo] {Hyunjin Lee, Young Jin Suh, and Changhwa Woo}

\address{\newline
Hyunjin Lee
\newline The Research Institute of Real and Complex Manifolds (RIRCM),
\newline Kyungpook National University,
\newline Daegu 41566, Repulic of Korea}
\email{lhjibis@hanmail.net}

\address{\newline
Young Jin Suh
\newline Department of Mathematics \& RIRCM,
\newline Kyungpook National University,
\newline Daegu 41566, Repulic of Korea}
\email{yjsuh@knu.ac.kr}

\address{\newline
Changhwa Woo
\newline Department of Applied Mathematics,
\newline Pukyong National University,
\newline Busan, Republic of Korea}
\email{legalgwch@pknu.ac.kr}

\footnotetext[2]{{\it Key words and phase} : Hopf real hypersurface, complex two-plane Grassmannians, quadratic Killing tensor, cyclic parallelism, structure Jacobi operator.  }

\thanks{* The first author was supported by grant Proj. No. NRF-2019-R1I1A1A-01050300, the
second by Proj. No. NRF-2018-R1D1A1B-05040381, and the third by  NRF-2020-R1A2C1A-01101518 from National Research Foundation of Korea.}

\begin{abstract}
In this paper, we introduce a notion of {\it quadratic Killing} structure Jacobi operator (simply, {\it Killing} structure Jacobi operator) and its geometric meaning for real hypersurfaces in the complex two-plane Grassmannians $\GBt$. In addition, we give a classification theorem for Hopf real hypersurfaces with {\it quadratic Killing structure Jacobi operator} in complex two-plane Grassmannians.
\end{abstract}

\maketitle

\section{introduction}\label{section 1}
\setcounter{equation}{0}
\renewcommand{\theequation}{1.\arabic{equation}}

In the class of complex Grassmannians of rank~2, usually we can give the example of a Hermitian symmetric space $G_{2}(\mathbb C^{m+2})=SU_{m+2}/S(U_2U_m)$, which is said to be {\it complex two-plane Grassmannians} of compact type. It is viewed as Hermitian symmetric spaces and quaternionic K\"{a}hler symmetric spaces equipped with the K\"{a}hler structure~$J$ and the quaternionic K\"{a}hler structure~$\mathcal J=\mathrm{span}\{J_{1}, J_{2}, J_{3}\}$ (see \cite{KoNo}, \cite{Suh2012}, and \cite{Suh2013}).

\vskip 6pt

Recall that a non-zero vector field~$X$ of Hermitian symmetric spaces $(\bar M, g)$ of rank~2 is called {\it singular} if it is tangent to more than one maximal flat in~$\bar M$. In particular, there are exactly two types of singular tangent vectors~$X$ of $G_{2}(\mathbb C^{m+2})$ which are characterized by the geometric properties $JX \in {\mathcal J}X$ and $JX \perp {\mathcal J}X$ (see \cite{B01} and \cite{BS1}).

\vskip 6pt

The Riemannian curvature tensor $\bar R$ of $\GBt$ is locally given by
\begin{equation}\label{eq: 1.1}
\begin{split}
& \bar R (X,Y)Z \\
&= g(Y,Z)X - g(X,Z)Y + g(JY, Z) JX  -  g(JX,Z)JY - 2g(JX, Y)JZ  \\
& \quad \quad + \sum_{\nu=1}^{3}\big \{ g(J_{\nu}Y, Z) J_{\nu} X - g(J_{\nu}X, Z) J_{\nu}Y - 2g(J_{\nu}X, Y) J_{\nu}Z \big \} \\
\end{split}
\end{equation}
\begin{equation*}
\begin{split}
& \quad \quad + \sum_{\nu=1}^{3}\big \{ g(J_{\nu}J Y, Z) J_{\nu} JX - g(J_{\nu}JX, Z) J_{\nu}JY \big \},
\end{split}
\end{equation*}
where $J_{1}$, $J_{2}$, $J_{3}$ is any canonical local basis of $\mathcal J$ and the tensor $g$ of type (0,2) stands for the Reimannian metric on complex two-plane Grassmannians $\GBt$ (see \cite{B01}, \cite{BS1}, and~\cite{BS2020}).

\vskip 6pt

For real hypersurfaces $M$ in complex two-plane Grassmannians $\GBt$, we have the following two natural geometric conditions: {\it the $1$-dimensional distribution $\mathcal C^{\bot}= \mathrm{span} \{\xi \}$ and the $3$-dimensional distribution $\mathcal {Q}^{\bot} = \mathrm{span}\{\xi_{1},\xi_{2}, \xi_{3}\}$ are invariant under the shape operator~$A$ of~$M$}. Here the almost contact structure vector field $\xi$ defined by $\xi = -JN$ is said to be a {\it Reeb} vector field, where $N$ denotes a local unit normal vector field of $M$ in complex two-plane Grassmannians $\GBt$, $m \geq 3$. The {\it almost contact 3-structure} vector fields $\xi_{1},\xi_{2},\xi_{3}$ spanning the 3-dimensional distribution $\mathcal{Q}^{\bot}$ of $M$ in $\GBt$ are defined by $\xi_{\nu} = - J_{\nu} N$ $(\nu=1, 2, 3$), where~$J_{\nu}$ denotes a canonical local basis of the quaternionic K\"{a}hler structure~$\mathcal J$, such that $T_{p}M = \mathcal Q \oplus \mathcal Q^{\bot} = \mathcal C \oplus \mathcal C^{\bot}$, $p \in M$. By using these invariant conditions for two kinds of distributions $\mathcal C^{\bot}$ and $\mathcal Q^{\bot}$ in $T_{p} \GBt$, Berndt and Suh gave the complete classification of real hypersurfaces in complex two-plane Grassmannians $\GBt$ as follows:
\begin{thm A}
[\cite{BS1}]\label{BS 1999}
Let $M$ be a connected real hypersurface in complex two-plane Grassmannians~$\GBt$, $m \geq 3$. Then both  $\mathcal C^{\bot}$ and  $\QP$ are invariant under the shape operator~$A$ of $M$ if and only if
\begin{enumerate}
\item [$(\mathcal T_{A})$] $M$ is an open part of a tube around a totally geodesic $\GBo$ in $\GBt$, or
\item [$(\mathcal T_{B})$]$m$ is even, say $m=2n$, and $M$ is an open part of a tube around a totally geodesic ${\mathbb H}P^n$ in $\GBt$.
\end{enumerate}
\end{thm A}

On the other hand, the Reeb vector field~$\xi$ is said to be {\it Hopf} if it is invariant under the shape operator~$A$, that is, $A\xi \in \mathcal C^{\bot}$. The 1-dimensional foliation of $M$ by the integral curves of the Reeb vector field $\xi$ is said to be a {\it Hopf foliation} of~$M$. We say that $M$ is a {\it Hopf real hypersurface} in $\GBt$ if and only if the Hopf foliation of $M$ is totally geodesic. By the almost contact metric structure $(\phi, \xi, \eta, g)$ and the formula $\nabla_{X}\xi=\phi AX$ for any $X \in TM$, it can be easily checked that $M$ is Hopf if and only if the Reeb vector field $\xi$ is Hopf, where the structure tensor $\p$ and the contact 1-form $\eta$ are defined by
$$
JX = \phi X + \eta(X) N, \ \ \mathrm{and} \ \  \eta(X)=g(X,\xi)
$$
respectively. In addition, when the distribution~$\mathcal Q^{\bot}$ of $M$ in complex two-plane Grassmannians $\GBt$ is invariant under the shape operator, we say that $M$ is a {\it $\mathcal Q^{\bot}$-invariant hypersurface}.

\vskip 6pt

Moreover, we say that the Reeb flow of $M$ in complex two-plane Grassmannians $\GBt$ is {\it isometric}, when the Reeb vector field~$\xi$ of $M$ is Killing. It implies that the metric tensor~$g$ of $M$ is invariant under the Reeb flow of $\x$, that is, $\mathcal L_{\xi}g =0$ where $\mathcal L_{\xi}$ is the Lie derivative along the direction of $\xi$. Related to this notion, for complex two-plane Grassmannians~$\GBt$ Berndt and Suh gave a remarkable characterization for real hypersurface of type~$(\mathcal T_{A})$ mentioned in Theorem~$\rm A$ (see~\cite{BS2}).

\vskip 6pt

Indeed, the notion of isometric Reeb flow is regarded as a typical example of Killing vector fields which are classical objects of differential geometry. As mentioned above, Killing vector fields are defined by vanishing of the Lie derivative of metric tensor~$g$ with respect to a vector~$X$, that is, $\mathcal L_{X}g=0$. On the other hand, applying such definition to certain tensors of spacetimes, we can be expressed  geometrical symmetries of the spacetimes. In general, the spacetime symmetries are used in the study of exact solutions of Einstein's field equations in general relativity. As an example of this study, in \cite{SCP} the spacetime fulfilling Einstein's field equations with vanishing of pseudo-quisi-conformal curvature tensor was considered and existence of Killing and conformal Killing vectors on such spacetime have been established.

\vskip 6pt

Now let us consider a generalization of such a Killing vector field on $(\bar M, g)$. A symmetric tensor field $\mathcal K$ of type (0,2) on $(\bar M, g)$ is said to be {\it Killing}, if it satisfies
$$
(\nabla_{X} \mathcal K) (Y, Z) + (\nabla_{Y} \mathcal K)(Z, X) + (\nabla_{Z} \mathcal K)(X,Y) =0
$$
for any vector fields $X$, $Y$, and $Z$. In general, a tensor field $\mathcal T$ of type $(0,k)$ on $(\bar M, g)$ is called {\it Killing tensor} if the complete symmetrization of $\nabla \mathcal T$ vanishes. This is equivalent to $(\nabla_{X}\mathcal T)(X, X, \cdots, X) = 0$. It follows again that for such a Killing tensor, the expression $\mathcal T (\dot{\gamma}, \dot{\gamma}, \cdots, \dot{\gamma})$ is constant along any geodesic $\gamma$ (see \cite{Semm03}). In particular, the existing literature on symmetric Killing tensors is huge, especially coming from theoretical physics (see  \cite{HMS} and \cite{Semm03}).

\vskip 6pt

As examples of such a symmetric Killing tensor, real hypersurfaces in complex two-plane Grassmannians $\GBt$ with {\it generalized Killing shape operator} were considered by Lee and Suh (see \cite{LS2020}). Recently, in \cite{S2020} Suh gave a classification for Hopf real hypersurfaces with {\it generalized Killing Ricci tensor} in $\GBt$. In the study of Generalized Robertson-Walker spacetimes (shortly, GRW spacetimes) in Lorentzian manifolds, the physical meaning of Killing tensor is used to be a perfect fluid spacetime for GRW (see \cite{MMSS}, \cite{REB}, and \cite{ShGh}).

\vskip 6pt

Motivated by such a symmetric Killing tensor, we consider the structure Jacobi operator $\R$, which is a symmetric tensor field of type (1,1) on $M$ in $\GBt$. It is said to be {\it quadratic Killing} structure Jacobi operator or {\it Killing} structure Jacobi operator (see \cite{MUS}, \cite{MUSM} and \cite{ShGh}), if the structure Jacobi operator~$\R$ satisfies
$$
g\big(({\nabla}_X \R)Y, Z \big) + g\big(({\nabla}_Y \R)Z, X \big) +g\big(({\nabla}_Z \R) X, Y \big) = 0
$$
for any $X$, $Y$, and $Z \in T_{p}M$, $ p \in M$. This cyclic parallelism of $\R$ is equivalent to
$$
g\big(({\nabla}_X \R) X, X \big)=0
$$
for any $X \in T_p M$, $p \in M$, by virtue of the linearization. Moreover, we give the geometric meaning of quadratic Killing structure Jacobi operator as follows: {\it Let $\gamma$ be any geodesic curve on $M$ such that $\gamma(0)=p$ and $\dot \gamma (0) = X$ as the initial conditions. Then the structure Jacobi curvature ${\mathbb R}_{\xi} (\dot \gamma , {\dot \gamma}) := g( \R(\dot \gamma) , \dot \gamma)$ is constant along the geodesic $\gamma$ of the vector field $X$}. Here we denotes ${\mathbb R}_{\xi}$ the structure Jacobi tensor of type (0,2) defined by ${\mathbb R}_{\xi}(X,Y) = g(\R X, Y)$ for any $X$ and $Y \in T_{p}M$, $p \in M$, which is symmetric (see Lemma~2.8 in \cite{Semm03}).

\vskip 6pt

From the assumption of the quadratic Killing structure Jacobi operator, first we assert that the unit normal vector field $N$ becomes singular as follows:
\begin{thm 1}\label{Theorem 1}
Let $M$ be a Hopf real hypersurface in complex two-plane Grassmannians $\GBt$ for $m \geq 3$. If $M$ has a quadratic Killing structure Jacobi operator, then the normal vector field~$N$ of $M$ is singular.
\end{thm 1}

Next, by using Theorem~1 we give a classification of Hopf real hypersurfaces in complex two-plane Grassmannians $\GBt$, $m \geq 3$, with quadratic Killing structure Jacobi operator as follows:
\begin{thm 2}\label{Theorem 2}
Let $M$ be a Hopf real hypersurface in complex two-plane Grassmannians $\GBt$, $m \geq 3$. Then the structure Jacobi operator~$\R$ of $M$ is quadratic Killing if and only if $M$ is locally congruent to an open part of a tube of $r = \frac{\pi}{4 \sqrt{2}}$ around a totally geodesic $\GBo$ in $\GBt$.
\end{thm 2}

\vskip 17pt

\section{Preliminaries}\label{section 2}
\setcounter{equation}{0}
\renewcommand{\theequation}{2.\arabic{equation}}

As mentioned in the introduction, the complete classifications of real hypersurfaces in complex two-plane Grassmannians $\GBt$, $m \geq 3$, satisfying two invariant conditions for the distributions $\mathcal C^{\bot}=\mathrm{span}\{\xi\}$ and $\QP=\mathrm{span}\{\xi_{1}, \xi_{2}, \xi_{3}\}$ was given in~\cite{BS1}.

\vskip 3pt

In fact, in \cite{B01} and \cite{BS1} Berndt and Suh gave the characterizations of the singular unit normal vector~$N$ of $M$ in complex two-plane Grassmannians~$\GBt$: \emph{There are two types of singular normal vector, those $N$ for which $JN \bot \mathcal J N$, and those for which $JN \in \mathcal J N$}. In other words, it means that $\xi \in \Q$ or $\xi \in \QP$ because $JN=-\xi$, $\mathcal JN=\mathrm{span}\{\xi_{1}, \xi_{2}, \xi_{3}\}=\QP$, and $TM=\Q \oplus \QP$. The following proposition tell us that the normal vector field $N$ on the model spaces of $(\mathcal T_{A})$ is singular of type of $JN \in \mathcal J N$, that is, $\xi \in \QP$.

\begin{pro A}[\cite{BS1}]\label{proposition A}
Let $(\mathcal T_{A})$ be the tube of radius $0 < r < \frac{\pi}{\sqrt{8}}$ around the totally geodesic $\GBo$ in $\GBt$. Then the following statements hold:
\begin{enumerate}[\rm (i)]
\item {$(\mathcal T_{A})$ is a Hopf hypersurface.}
\item {Every unit normal vector field $N$ of $(\mathcal T_{A})$ is singular and of type $JN \in \mathcal J N$.}
\item {The eigenvalues and their corresponding eigenspaces and multiplicities are given in Table~\ref{table 1}.
\begin{table}[ht]
\caption{\footnotesize Principal curvatures of a model space of type $(\mathcal T_{A})$}\label{table 1}
\begin{tabular}{p{0.8cm}p{3.2cm}p{5.5cm}p{1.5cm}}
\hline\noalign{\smallskip}
\footnotesize{Type} & \footnotesize{Eigenvalues} & \footnotesize{Eigenspace} & \footnotesize{Multiplicity}  \\
\noalign{\smallskip}\hline\noalign{\smallskip}
$ \scriptstyle (\mathcal T_{A})$ & $ \scriptstyle \alpha = \sqrt{8}\cot(\sqrt{8}r)$  &  $\scriptstyle T_{\al}=\mathcal C^{\bot}=\mathrm{span}\{\xi \} = \mathrm{span}\{\xi_{1} \}$ & $\scriptstyle 1$     \\
    $  $           & $ \scriptstyle \beta = \sqrt{2} \cot(\sqrt{2}r)$  & $ \scriptstyle T_{\beta}=\mathcal C \ominus \mathcal Q = \mathrm{span}\{\xi_{2}, \xi_{3} \}$   & $\scriptstyle 2$  \\
                  & $ \scriptstyle \lambda = -\sqrt{2}\tan(\sqrt{2}r)$ & $ \scriptstyle T_{\lambda}=E_{-1}= \{X \in \, \Q \,| \, \phi X= \phi_{1}X \}$ &  $\scriptstyle 2m-2$ \\
                  & $\scriptstyle \mu=0$             & $ \scriptstyle T_{\mu}=E_{+1}=\{X \in \, \Q \,| \phi X= -\phi_{1}X \}$ &   $\scriptstyle 2m-2$\\
\noalign{\smallskip}\hline \noalign{\smallskip}
\end{tabular}
\end{table}
}
\item[\rm (iv)] {The Reeb flow on $(\mathcal T_{A})$ is isometric.}
\end{enumerate}
\end{pro A}

On the other hand, by using the notion of isometric Reeb flow, that is, the shape operator~$A$ of a Hopf real hypersurface~$M$ in $\GBt$ commutes with structure tensor $\phi$, that is, $A \phi = \phi A$, Berndt and Suh gave:
\begin{equation}\label{e: 2.1}
\begin{split}
(\nabla_{X}A)Y &  = - \eta(Y) \phi X +(X\alpha)\eta(Y) \x + \alpha g(A \phi X, Y) \xi - g(A^{2} \phi X, Y) \xi \\
& \quad  - \sum_{i=1}^{3}\big\{ \en(Y)\pn X + g(\pn \xi, Y) \p \pn X + 2g(\pn \x, X) \p \pn Y  \\
& \quad \quad \quad  \quad  + g(\pn \xi, X) \en (Y) \xi - \en(\xi) g(\pn X, Y) \xi + g(\pn X, Y)\xn  \\
& \quad \quad \quad  \quad  - \e(X)\en(Y) \pn \x + g(\pn \p X, Y) \pn \x  \big \}
\end{split}
\end{equation}
for any tangent vector fields $X$ and $Y$ on $M$ (see Proposition~4 in \cite{BS2}). In fact, from (iv) in Proposition~$\rm A$, we see that the shape operator~$A$ of $(\mathcal T_{A})$ satisfies $A \phi = \phi A$. Thus, the above equation \eqref{e: 2.1} holds on $(\mathcal T_{A})$ and it can be rearranged as
\begin{equation}\label{e: 2.2}
\begin{split}
(\nabla_{X}A)Y &  = - \eta(Y) \phi X + \alpha g(A \phi X, Y) \xi - g(A^{2} \phi X, Y) \xi \\
& \quad \,  - \sum_{i=1}^{3}\big\{ \en(Y)\pn X + g(\pn \xi, Y) \p \pn X + 2g(\pn \x, X) \p \pn Y  \\
& \quad \quad \quad  \quad  + g(\pn \xi, X) \en (Y) \xi - \en(\xi) g(\pn X, Y) \xi + g(\pn X, Y)\xn  \\
& \quad \quad \quad  \quad  - \e(X)\en(Y) \pn \x + g(\pn \p X, Y) \pn \x  \big \}
\end{split}
\end{equation}
for any tangent vector fields $X$ and $Y$ on $T(\mathcal T_{A})=T_{\alpha}\oplus T_{\beta}\oplus T_{\lambda}\oplus T_{\mu}$.
\vskip 17pt

\section{Fundamental equations of real hypersrufaces in $\GBt$}\label{section 2-1}
\setcounter{equation}{0}
\renewcommand{\theequation}{3.\arabic{equation}}

We use some references~\cite{LS}, \cite{PeSu}, and \cite{S2013} to recall the Riemannian geometry of complex two-plane Grassmannians~$\GBt$, $m \geq 3$, and some fundamental formulas including the Codazzi and Gauss equations for a real hypersurface in $\GBt$.

\vskip 3pt

Let $M$ be a real hypersurface of complex two-plane Grassmannians~$\GBt$, $m \geq 3$, that is, a submanifold of $\GBt$ with real codimension one. The induced Riemannian metric on $M$ will also be denoted by $g$, and $\nabla$ denotes the Riemannian connection of $(M,g)$. Let $N$ be a local unit normal field of $M$ in $\GBt$ and $S$ the shape operator of $M$ with respect to $N$, that is, ${\bar \nabla}_{X}N = -SX$. The K\"{a}hler structure $J$ of complex two-plane Grassmannians~$\GBt$ induces on $M$ an almost contact metric structure $(\phi,\xi,\eta,g)$. Furthermore, let $J_1$, $J_2$, $J_3$ be a canonical local basis of the quaternionic K\"{a}hler structure~${\mathcal J}$. Then each $J_\nu$ induces an almost contact metric structure $(\phi_\nu,\xi_\nu,\eta_\nu,g)$ on $M$. Now let us put
\begin{equation}\label{eq: 3.1}
JX={\Ph}X+{\eta}(X)N,\quad J_{\nu}X={\Ph}_{\nu}X+{\eta}_{\nu}(X)N
\end{equation}
for any tangent vector $X$ of a real hypersurface $M$ in $\GBt$, where $N$ denotes a normal vector of $M$ in $\GBt$. Then the following identities can be proved in a straightforward method and will be used frequently in subsequent
calculations:
\begin{equation}\label{eq: 3.2}
\begin{split}
&{\phi}_{\nu +1}{\xi}_{\nu}=-{\xi}_{{\nu}+2},\quad {\phi}_{\nu}{\xi}_{{\nu}+1}={\xi}_{{\nu}+2}, \quad {\phi}{\xi}_{\nu}={\phi}_{\nu}{\xi},\quad {\eta}_{\nu}({\phi}X)={\eta}({\phi}_{\nu}X),\\
&{\phi}_{\nu}{\phi}_{{\nu}+1}X={\phi}_{{\nu}+2}X+{\eta}_{{\nu}+1}(X){\xi}_{\nu}, \quad {\phi}_{{\nu}+1}{\phi}_{\nu}X=-{\phi}_{{\nu}+2}X+{\eta}_{\nu}(X){\xi}_{{\nu}+1},
\end{split}
\end{equation}
where we have used that $J_{\nu}J_{\nu+1}=J_{\nu+2}= - J_{\nu+1}J_{\nu}$.

\vskip 3pt

On the other hand, from the parallelism of $J$ and $\mathcal J$ which are defined by
\begin{equation*}
\bar \nabla_{X} J =0 \ \ \mathrm{and} \ \ \bar \nabla_{X}J_{\nu} = q_{\nu+2}(X)J_{\nu+1} - q_{\nu+1}(X) J_{\nu+2} \ (\nu \ \mathrm{mod}\ 3),
\end{equation*}
together with Gauss and Weingarten formulas, it follows that
\begin{equation}\label{eq: 3.3}
({\N}_X{\Ph})Y={\E}(Y)AX-g(AX,Y){\xi},\quad {\nabla}_X{\xi}={\Ph}AX,
\end{equation}
\begin{equation}\label{eq: 3.4}
{\N}_X{\xi}_{\nu}=q_{{\nu}+2}(X){\xi}_{{\nu}+1}-q_{{\nu}+1}(X){\xi}_{{\nu}+2}
+{\Ph}_{\nu}AX,
\end{equation}
\begin{equation}\label{eq: 3.5}
\begin{split}
({\N}_X{\Ph}_{\nu})Y & =-q_{{\nu}+1}(X){\Ph}_{{\nu}+2}Y+q_{{\nu}+2}(X){\Ph}_{{\nu}+1}Y\\
& \quad \  +{\E}_{\nu}(Y)AX -g(AX,Y){\xi}_{\nu}.
\end{split}
\end{equation}
Combining these formulas, we find the following
\begin{equation}\label{eq: 3.6}
\begin{split}
{\N}_X({\Ph}_{\nu}{\xi}) &={\N}_X({\Ph}{\xi}_{\nu}) \\
& =({\N}_X{\Ph}){\xi}_{\nu}+{\Ph}({\N}_X{\xi}_{\nu})\\
& =q_{{\nu}+2}(X){\Ph}_{{\nu}+1}{\xi}-q_{{\nu}+1}(X){\Ph}_{{\nu}+2}{\xi}+{\Ph}_{\nu}{\Ph}AX -g(AX,{\xi}){\xi}_{\nu}+{\E}({\xi}_{\nu})AX.
\end{split}
\end{equation}
Moreover, from $JJ_{\nu}=J_{\nu}J$, ${\nu}=1,2,3$, it follows that
\begin{equation}\label{eq: 3.7}
{\Ph}{\Ph}_{\nu}X={\Ph}_{\nu}{\Ph}X+{\eta}_{\nu}(X){\xi}-{\eta}(X){\xi}_{\nu}.
\end{equation}

\vskip 3pt

Finally, using the explicit expression for the Riemannian curvature tensor $\bar{R}$ of complex two-plane Grassmannians $\GBt$ in the introduction, the Codazzi and Gauss equations of $M$ in $\GBt$, are given respectively by
\begin{equation}\label{eq: 3.8}
\begin{split}
(\nabla_XA)Y - (\nabla_YA)X
& = \eta(X)\phi Y - \eta(Y)\phi X - 2g(\phi X,Y)\xi \\
& \quad + \sum_{\nu=1}^3 \big\{\eta_\nu(X)\phi_\nu Y -
\eta_\nu(Y)\phi_\nu
X - 2g(\phi_\nu X,Y)\xi_\nu\big\} \\
& \quad + \sum_{\nu=1}^3 \big\{\eta_\nu(\phi X)\phi_\nu\phi Y
- \eta_\nu(\phi Y)\phi_\nu\phi X\big\} \\
& \quad + \sum_{\nu=1}^3 \big\{\eta(X)\eta_\nu(\phi Y) -
\eta(Y)\eta_\nu(\phi X)\big\}\xi_{\nu}
\end{split}
\end{equation}
and
\begin{equation}\label{eq: 3.9}
\begin{split}
& R(X,Y)Z - g(AY,Z)AX + g(AX,Z)AY \\
&= g(Y,Z)X - g(X,Z)Y + g({\phi}Y,Z){\phi}X - g({\phi}X,Z){\phi}Y - 2g({\phi}X,Y){\phi}Z\\
&\qquad + \SN \big \{g(\PN Y,Z)\PN X - g(\PN X,Z)\PN Y - 2g(\PN X,Y)\PN Z \big \} \\
&\qquad + \SN \big\{g(\PN{\phi}Y,Z)\PN {\phi}X - g(\PN{\phi}X,Z){\PN}{\phi}Y\big\}\\
&\qquad + \SN \big\{{\eta}(X){\EN}(Z){\PN}{\phi}Y - {\eta}(Y){\EN}(Z){\PN}{\phi}X\big\}\\
&\qquad + \SN \big\{{\eta}(Y)g({\PN}{\phi}X,Z) - {\eta}(X)g({\PN}{\phi}Y,Z) \big\}{\KN}
\end{split}
\end{equation}
for any tangent vector fields~$X$, $Y$ and $Z$ on $M$.

\vskip 3pt

On the other hand, we can derive some important facts from the geometric condition of $M$ being Hopf, that is, $A\xi = \al \xi$ where $\al = g(A\xi, \xi)$. Among them, we introduce the following formulas which are induced from the Codazzi equation:
\begin{lem A}[\cite{BS2}]
If $M$ is a connected orientable Hopf real hypersurface in complex two-plane Grassmannians $\GBt$, $m \geq 3$, then
\begin{equation}\label{eq: 3.10}
\mathrm{grad} \, \al = (\x \al)\x + 4\SN \en(\x)\p_{\nu} \x
\end{equation}
and
\begin{equation}\label{eq: 3.11}
\begin{split}
& 2 A \phi A X - \alpha A\phi X  - \al \phi A X \\
& \quad =  2 \phi X +  2 \SN \big\{ \en (X) \p_{\nu} \x - g(\phi_{\nu}\xi, X) \xi_{\nu} +
\en (\xi) \pn X  \big\} \\
\end{split}
\end{equation}
\begin{equation*}
\begin{split}
& \quad \quad  \quad -4 \SN  \big \{ \eta(X) \en (\xi) \pn \x - \en(\xi) g(\pn \xi, X)  \x  \big\},
\end{split}
\end{equation*}
for any tangent vector field $X$ on $M$ in $\GBt$.
\end{lem A}

\vskip 17pt

\section{Proof of Theorem~1}\label{section 4}

\setcounter{equation}{0}
\renewcommand{\theequation}{4.\arabic{equation}}

Let $M$ be a Hopf real hypersurface with quadratic Killing structure Jacobi operator in complex two-plane Grassmannians~$\GBt$, $m \geq 3$. From the notion of Killing structure Jacobi operator~$R_{\xi}$ the equivalent condition, so-called, {\it cyclic parallel structure Jacobi operator}~$\R$ can be given by
\begin{equation}\label{eq: Killing structure Ja op}
\begin{split}
& \mathfrak S_{X,Y,Z \in TM}\, g\big((\nabla_{X}\R)Y, Z \big) \\
& \ \  = g\big((\nabla_{X}\R)Y, Z \big) + g\big((\nabla_{Y}\R)Z, X \big)+g\big((\nabla_{Z}\R)X, Y \big)=0,
\end{split}
\tag{\dag}
\end{equation}
for any tangent vector fields $X$, $Y$, and $Z$ on $M$. The formula~\eqref{eq: Killing structure Ja op} is said to be {\it quadratic Killing} (or simply, {\it Killing}) structure Jacobi operator.

\vskip 6pt

From \eqref{eq: 3.9} the structure Jacobi operator~$\R \in \mathrm{End}(TM)$ is given as follows
\begin{equation}\label{eq: 4.1 structure Jacobi operator}
\begin{split}
R_{\xi}(Y) & = R(Y,\xi)\xi \\
& = Y - \eta(Y) \xi + \alpha AY - \alpha^{2} \e(Y) \xi \\
& \quad \  - \sum_{\nu=1}^{3} \big\{ \en(Y) \xn - \e(Y) \en(\xi) \xn - 3g(\pn \xi, Y) \pn \xi + \en(\xi) \pn \p Y \big \}
\end{split}
\end{equation}
for any tangent vector field $Y \in TM$ (see \cite{LS2017} and \cite{MPS 2015}).

\vskip 3pt

Taking the covariant derivative of \eqref{eq: 4.1 structure Jacobi operator} along the direction of $X$ implies
\begin{equation}\label{eq: 4.2}
\begin{split}
(\nabla_{X}R_{\xi})Y & = \nabla_{X}(\R Y) - \R (\nabla_{X}Y) \\
& = -g(\p AX, Y) \xi - \e(Y) \p AX \\
& \quad   - \sum_{\nu=1}^{3} \Big [ g(\pn AX, Y) \xn + 2\e(Y) g(\pn \x, AX) \xn  + \en(Y) \pn AX \\
& \quad \quad \quad \ \  +3 g(\pn AX, \p Y) \pn \xi + 3 \e(Y) \en(AX) \pn \x \\
& \quad \quad \quad \ \  - 3g(\pn \x, Y) \pn \p AX   + 3 \alpha \e (X) g(\pn \xi, Y) \xn \\
& \quad \quad \quad \ \  - 4 \en(\x) g(\pn \xi, Y) AX  - 4 \en(\x)g(AX, Y) \pn \x \\
& \quad \quad \quad \ \  - 2g(\pn \x, AX) \pn\p Y \Big ] \\
&\quad  + g((\nabla_{X}A)\xi, \xi)AY + \alpha (\nabla_{X}A)Y - \alpha g((\nabla_{X}A)Y, \xi)\xi \\
&\quad   - \alpha g(AY, \phi AX) \xi - \alpha \eta(Y) (\nabla_{X}A) \xi - \alpha \eta(Y) A \phi AX
\end{split}
\end{equation}
for any tangent vector fields $X$ and $Y$ on $M$ (see \cite{LS2017}). From this and using symmetric property of the structure Jacobi operator~$\R$ in $\GBt$, the quadratic Killing structure Jacobi operator~\eqref{eq: Killing structure Ja op} can be rearranged as follows:
\begin{equation}\label{eq: 4.3}
\begin{split}
0 &= g\big((\nabla_{X}\R)Y, Z \big) + g\big((\nabla_{Y}\R)Z, X \big)+g\big((\nabla_{Z}\R)X, Y \big)\\
& = g\big((\nabla_{X}\R)Y, Z \big) + g\big((\nabla_{Y}\R)X, Z \big) \\
&  \ \ -g(\p AZ, X) \eta(Y) - \e(X) g(\p AZ, Y) + g((\nabla_{Z}A)\xi, \xi)g(AX, Y) \\
& \ \ + \alpha g((\nabla_{Z}A)X, Y) - \alpha g((\nabla_{Z}A)X, \xi)\e(Y) \\
& \ \ - \alpha g(AX, \phi AZ) \e(Y) - \alpha \eta(X) g((\nabla_{Z}A) \xi, Y) - \alpha \eta(X) g(A \phi AZ, Y)\\
& \ \  - \sum_{\nu=1}^{3} \Big [ g(\pn AZ, X) \en(Y) + 2\e(X) g(\pn \x, AZ) \en(Y)  \\
& \quad \quad \quad \quad  + \en(X) g(\pn AZ, Y)  +3 g(\pn AZ, \p X) g(\pn \xi, Y) \\
& \quad \quad \quad \quad  + 3 \e(X) \en(AZ) g(\pn \x, Y)  - 3g(\pn \x, X) g(\pn \p AZ, Y)   \\
\end{split}
\end{equation}
\begin{equation*}
\begin{split}
& \quad \quad \quad \quad  + 3 \alpha \e (Z) g(\pn \xi, X) \en(Y)  - 4 \en(\x) g(\pn \xi, X) g(AZ, Y) \\
& \quad \quad \quad \quad   - 4 \en(\x)g(AZ, X) g(\pn \x, Y)  - 2g(\pn \x, AZ) g(\pn\p X, Y) \Big ] \\
& = g\big((\nabla_{X}\R)Y, Z \big) + g\big((\nabla_{Y}\R)X, Z \big) \\
&  \ \ +g(A\p X, Z) \eta(Y) + \e(X) g(A \p  Y, Z) + (\xi \alpha)  g(AX, Y)\eta(Z) \\
& \ \ - \alpha (\xi \alpha) \e(X) \e(Y) \eta(Z) + \alpha^{2} \e(Y) g(A \phi  X, Z) - \alpha \e(Y) g(A \phi A X, Z) \\
& \ \ + \alpha \e(Y) g(A \phi AX, Z)   + \alpha \eta(X) g(A \phi AY, Z)\\
& \ \ - \alpha (\xi \alpha) \eta(X)  \e(Y) \eta(Z)  + \alpha^{2} \eta(X) g(A \phi  Y, Z) - \alpha \eta(X) g(A \phi A Y, Z) \\
& \ \ + \alpha g((\nabla_X A)Y, Z)  + \alpha g( \phi X, Y)\eta(Z) + \alpha \eta(X)g(\phi Y,  Z) + 2\alpha \eta(Y) g(\phi X, Z) \\
& \ \  + \sum_{\nu=1}^{3} \Big [ \en(Y) g(A \pn X, Z)  - 2\e(X)\en(Y) g(A \pn \x, Z)   + \en(X) g(A \pn Y, Z) \\
& \quad \quad \quad \quad  + 3g(\pn \xi, Y) g(A \pn \p X, Z)  - 3 \e(X)  g(\pn \x, Y) g(A\xn, Z) \\
& \quad \quad \quad \quad  + 3g(\pn \x, X) g(A \p \pn Y, Z)  - 3 \alpha g(\pn \xi, X) \en(Y)\e(Z) \\
& \quad \quad \quad \quad  + 4 \en(\x) g(\pn \xi, X) g(AY, Z) + 4 \en(\x) g(\pn \x, Y)g(AX, Z)  \\
& \quad \quad \quad \quad  + 2 g(\pn\p X, Y)g(A \pn \x, Z) + 4 g(AX, Y)  \en(\xi) g(\pn \xi, Z)
\\
& \quad \quad \quad \quad
- 4 \alpha \e(X) \e(Y) \en(\xi) g(\pn \xi, Z)  - 4 \alpha \eta(X) \e(Y) \en(\xi) g(\pn \xi, Z) \Big ] \\
& \quad + \alpha \sum_{\nu=1}^3 \Big [ g(\phi_\nu X, Y)\eta_\nu(Z) + \eta_\nu(X)g(\phi_\nu Y, Z) +
 2\en(Y) g(\phi_\nu X,Z) \\
& \quad \quad \quad \quad
- g(\phi_\nu\phi X, Y)g(\pn \xi, Z) + g(\pn \xi, X)g(\phi \phi_\nu Y, Z) \\
& \quad \quad \quad \quad + \eta_\nu(\phi X)\en(Y) \eta(Z) + \eta(X)\en(Y) g(\pn \xi, Z)\Big ],
\end{split}
\end{equation*}
where we have used
\begin{equation*}
\begin{split}
g((\nabla_{Z}A)\xi, X) & = (Z \alpha) \e(X) - \alpha g(A \phi  X, Z) + g(A \phi A X, Z) \\
& = (\xi \alpha) \eta(Z) \e(X) + 4 \sum_{\nu=1}^{3} \en(\xi) g(\pn \xi, Z) \e(X)  \\
& \quad \ \ - \alpha g(A \phi X, Z) + g(A \phi A X, Z),
\end{split}
\end{equation*}
and
\begin{equation*}
\begin{split}
& g((\nabla_ZA)X, Y) \\
& = g((\nabla_X A)Z, Y)  + \eta(Z)g( \phi X, Y) - \eta(X)g(\phi Z, Y) - 2g(\phi Z,X)\eta(Y) \\
& \quad + \sum_{\nu=1}^3 \big\{\eta_\nu(Z)g(\phi_\nu X, Y) - \eta_\nu(X)g(\phi_\nu
Z, Y) - 2g(\phi_\nu Z,X)\en(Y) \big\} \\
& \quad + \sum_{\nu=1}^3 \big\{\eta_\nu(\phi Z)g(\phi_\nu\phi X, Y)
- \eta_\nu(\phi X)g(\phi_\nu\phi Z, Y) \big\} \\
\end{split}
\end{equation*}
\begin{equation*}
\begin{split}
& \quad + \sum_{\nu=1}^3 \big\{\eta(Z)\eta_\nu(\phi X) -
\eta(X)\eta_\nu(\phi Z)\big\}\en(Y)
\end{split}
\end{equation*}
for any tangent vector fields~$X$, $Y$, and $Z$ on $M$. Deleting $Z$ from \eqref{eq: 4.3} and using \eqref{eq: 4.2} gives
\begin{equation}\label{eq: 4.4}
\begin{split}
&  -g(\p AX, Y) \xi - \e(Y) \p AX -g(\p AY, X) \xi - \e(X) \p AY  + \eta(Y)A\p X \\
&  + \e(X) A \p  Y + (\xi \alpha)  g(AX, Y) \xi - 2 \alpha (\xi \alpha) \e(X) \e(Y) \xi + \alpha^{2} \e(Y) A \phi  X   \\
&  + \alpha^{2} \eta(X) A \phi Y + \alpha (\nabla_X A)Y  + \alpha g(\phi X, Y)\xi + \alpha \eta(X) \phi Y + 2\alpha \eta(Y) \phi X \\
&    - \sum_{\nu=1}^{3} \Big [ g(\pn AX, Y) \xn + 2\e(Y) g(\pn \x, AX) \xn  + \en(Y) \pn AX \\
& \quad \quad \quad   +3 g(\pn AX, \p Y) \pn \xi + 3 \e(Y) \en(AX) \pn \x - 3g(\pn \x, Y) \pn \p AX  \\
& \quad \quad \quad  + 3 \alpha \e (X) g(\pn \xi, Y) \xn - 4 \en(\x) g(\pn \xi, Y) AX \\
& \quad \quad \quad - 4 \en(\x)g(AX, Y) \pn \x - 2g(\pn \x, AX) \pn\p Y \\
& \quad \quad \quad  + g(\pn AY, X) \xn + 2\e(X) g(\pn \x, AY) \xn  + \en(X) \pn AY \\
& \quad \quad \quad  +3 g(\pn AY, \p X) \pn \xi + 3 \e(X) \en(AY) \pn \x - 3g(\pn \x, X) \pn \p AY  \\
& \quad \quad \quad  + 3 \alpha \e (Y) g(\pn \xi, X) \xn - 4 \en(\x) g(\pn \xi, X) AY \\
& \quad \quad \quad - 4 \en(\x)g(AY, X) \pn \x - 2g(\pn \x, AY) \pn\p X \Big ] \\
& + \sum_{\nu=1}^{3} \Big [ \en(Y) A \pn X  - 2\e(X)\en(Y) A \pn \x + \en(X) A \pn Y \\
& \quad \quad \quad  + 3g(\pn \xi, Y) A \pn \p X  - 3 \e(X)  g(\pn \x, Y) A\xn \\
& \quad \quad \quad  + 3g(\pn \x, X) A \p \pn Y - 3 \alpha g(\pn \xi, X) \en(Y)\xi \\
& \quad \quad \quad  + 4 \en(\x) g(\pn \xi, X) AY + 4 \en(\x) g(\pn \x, Y) AX \\
& \quad \quad \quad  + 2 g(\pn\p X, Y) A \pn \x  + 4 g(AX, Y)\en(\xi) \pn \xi
 \\
& \quad \quad \quad - 4 \alpha \e(X) \e(Y) \en(\xi) \pn \xi - 4 \alpha \eta(X) \e(Y) \en(\xi) \pn \xi \Big ] \\
& + \alpha \sum_{\nu=1}^3 \Big [ g(\phi_\nu X, Y) \xn + \eta_\nu(X) \phi_\nu Y +
 2\en(Y) \phi_\nu X - g(\phi_\nu\phi X, Y) \pn \xi\\
& \quad \quad \quad  \ \  + g(\pn \xi, X) \phi \phi_\nu Y + \eta_\nu(\phi X)\en(Y) \xi +
\eta(X)\en(Y) \pn \xi \Big ]\\
& + g((\nabla_{X}A)\xi, \xi)AY  - \alpha g((\nabla_{X}A)Y, \xi)\xi - \alpha g(AY, \phi AX) \xi \\
& - \alpha \eta(Y) A \phi AX  + g((\nabla_{Y}A)\xi, \xi)AX  - \alpha g((\nabla_{Y}A)X, \xi)\xi  \\
& - \alpha g(AX, \phi AY) \xi - \alpha \eta(X) A \phi AY + \alpha (\nabla_{X}A)Y - \alpha \eta(Y) (\nabla_{X}A) \xi \\
& + \alpha (\nabla_{Y}A)X  - \alpha \eta(X) (\nabla_{Y}A) \xi =0.
\end{split}
\end{equation}
On the other hand, by the Codazzi equation \eqref{eq: 3.8} and \eqref{eq: 3.10} in the latter part of \eqref{eq: 4.4}, we obtain
\begin{equation}\label{eq: 4.5}
\begin{split}
& g((\nabla_{X}A)\xi, \xi)AY  - \alpha g((\nabla_{X}A)Y, \xi)\xi - \alpha g(AY, \phi AX) \xi - \alpha \eta(Y) A \phi AX\\
&\  + g((\nabla_{Y}A)\xi, \xi)AX - \alpha g((\nabla_{Y}A)X, \xi)\xi - \alpha g(AX, \phi AY) \xi - \alpha \eta(X) A \phi AY\\
&\ + \alpha (\nabla_{X}A)Y + \alpha (\nabla_{Y}A)X  - \alpha \eta(Y) (\nabla_{X}A) \xi - \alpha \eta(X) (\nabla_{Y}A) \xi\\
& = (\x \alpha) \e(X) AY + 4 \sum_{\nu=1}^{3} \en(\x) g(\pn \x, X) AY - \alpha g(A\phi AX, Y) \xi \\
& \quad - \alpha \eta(Y) A \phi AX -\alpha (\xi \alpha) \eta(X) \e(Y)\xi  - 4\alpha \sum_{\nu=1}^{3} \en(\xi) g(\pn \x, X) \e(Y)\xi  \\
& \quad  - \alpha^{2} g(\p AX,Y) \xi + \alpha g(A \p AX, Y) \xi + (\x \alpha) \e(Y) AX \\
&  \quad + 4 \sum_{\nu=1}^{3} \en(\x) g(\pn \x, Y) AX + \alpha g(A\phi AX, Y) \xi \\
& \quad  - \alpha \eta(X) A \phi AY -\alpha (\xi \alpha) \e(X) \eta(Y) \xi \\
&  \quad  -4 \alpha \sum_{\nu=1}^{3} \en(\xi) \e(X) g(\pn \x, Y) \xi - \alpha^{2} g(\p AY,X)\xi + \alpha g(A \p AY, X) \xi \\
&  \quad + 2 \alpha (\nabla_X A)Y + \alpha \eta(Y)\phi X - \alpha \eta(X)\phi Y - 2\alpha g(\phi Y,X)\xi \\
&  \quad+ \alpha \sum_{\nu=1}^3 \big \{ \eta_\nu(Y)\phi_\nu X - \eta_\nu(X)\phi_\nu
Y - 2g(\phi_\nu Y,X)\xi_\nu  \eta_\nu(\phi Y)\phi_\nu\phi X \big \}\\
&  \quad + \alpha \sum_{\nu=1}^3 \big \{ - \eta_\nu(\phi X)\phi_\nu\phi Y + \eta(Y)\eta_\nu(\phi X)\xn - \eta(X)\eta_\nu(\phi Y)\xi_{\nu} \big \}\\
&  \quad-\alpha \eta(Y) \{(\xi \alpha) \e(X) \x + 4 \sum_{\nu=1}^{3} \en(\x) g(\pn \xi, X) \x + \alpha \p AX - A\p AX \big \}\\
&  \quad-\alpha \eta(X) \{(\xi \alpha) \e(Y) \x + 4 \sum_{\nu=1}^{3} \en(\x) g(\pn \xi, Y) \x + \alpha \p AY - A\p AY \big \}.\\
\end{split}
\end{equation}

From now on, we want to prove the normal vector field~$N$ of a Hopf real hypersurface~$M$ in $\GBt$ is singular. Then by the meaning of singular mentioned in the introduction, we see that either $\xi \in \Q$ or $\xi \in \QP$ where $\Q$ is the maximal quaternionic subbundle of $TM = \Q \oplus \QP$. In order to do this, we may put the Reeb vector field~$\xi$ as follows:
\begin{equation}\label{eq: Reeb vector}
\xi = \e(X_{0})\X + \e(\xo) \xo
\tag{*}
\end{equation}
for unit vector fields $\X \in \Q$ and $\xo \in \QP$ with $\e(\X) \e(\xo) \neq 0$. By using the notation~\eqref{eq: Reeb vector} we obtain that {\it the Reeb function~$\alpha$ is constant along the direction of $\xi$ if and only if the distribution~$\mathcal Q$- or the $\mathcal Q^{\bot}$-component of the structure vector field~$\xi$ is invariant by the shape operator, that is $A \X = \alpha \X$ and $A \xo = \alpha \xo$} (see \cite{JCPS2011} and \cite{LS2013}). From this fact, we obtain the following useful formulas for Hopf real hypersurfaces in $\GBt$.
\begin{lemma}\label{lem 4.1}
Let $M$ be a Hopf real hypersurface with non-vanishing geodesic Reeb flow in $\GBt$, $m \geq 3$. If the distribution $\mathcal Q$ or $\mathcal Q^{\bot}$ component of the structure vector field~$\xi$ by the shape operator, then the following formulas hold:
\begin{enumerate}[\rm (a)]
\item {$A \phi \X = \mu \p \X$},
\item {$A \p \xo = \mu \p \xo$, and}
\item {$A \po \X = \mu \po \X$}
\end{enumerate}
where the function $\mu$ is given by $\mu = \frac{\alpha^{2}+ 4 \eta^{2}(\X)}{\alpha}$.
\end{lemma}

\begin{proof}
Putting $X = \X$ in \eqref{eq: 3.11} and using $A\X = \alpha \X$, it yields
\begin{equation}\label{eq: 4.6}
\begin{split}
& \alpha A \phi \X =  \al^{2} \phi  \X + 2 \phi \X +  2 \e (\xo) \po \X -4 \eta(\X) \e (\xi_{1}) \po \x ,
\end{split}
\end{equation}
where we have used $g(\pn \xi, \X) =0$ for $\nu=1,2,3$ and $\e_{2}(\xi) = \e_{3}(\xi)=0$.

\vskip 3pt

On the other hand, by \eqref{eq: Reeb vector} we obtain
\begin{equation}\label{eq: 4.7}
\po \xi = \e(\X) \po \X + \e(\xo) \po \xo = \e(\X) \po \X.
\end{equation}
In addition, from \eqref{eq: Reeb vector} and $\phi_{1} \xi = \po \x$ we have
\begin{equation*}
\begin{split}
0  = \phi \xi & = \e(\X) \p \X + \e(\xo) \p \xo \\
& = \e(\X) \p \X + \e(\xo) \po \x \\
& = \e(\X) \p \X + \e(\xo) \e(\X) \po \X,
\end{split}
\end{equation*}
which means
\begin{equation}\label{eq: 4.8}
\p \X = - \e(\xo) \po \X
\end{equation}
because $\e(\X) \e(\xo)\neq 0$.  Substituting \eqref{eq: 4.7} and \eqref{eq: 4.8} to \eqref{eq: 4.6}, we get
\begin{equation*}
 \alpha A \phi \X  =  \al^{2} \phi  \X + 4 \eta^{2}(\X) \p \X = (\al^{2} + 4 \eta^{2}(\X)) \p \X.
\end{equation*}
Since $M$ has non-vanishing geodesic Reeb flow, we see that the vector field~$\phi \X$ is principal with corresponding principal curvature $\mu = \frac{\al^{2}+ 4 \eta^{2}(\X)}{\alpha}$.

\vskip 3pt

Similarly, using \eqref{eq: 4.7} and \eqref{eq: 4.8}, together with $\e(\X) \e(\xo) \neq 0$, the formula \eqref{eq: 4.6} gives (b) and~(c).
\end{proof}

When the Reeb function~$\alpha$ is vanishing, P\'{e}rez and Suh gave the following
\begin{lem B}[\cite{PeSu}]
Let $M$ be a Hopf real hypersurface in $\GBt$, $m \geq 3$. If $M$ has vanishing geodesic Reeb flow, then the unit normal vector field~$N$ of $M$ is singular, that is, either $\xi \in \mathcal Q$ or $\xi \in \mathcal Q^{\bot}$.
\end{lem B}
\begin{remark}\label{rem 4.2}{\rm
By using the method in the proof of Lemma~$\rm B$, we can assert that {\it if $M$ is a Hopf real hypersurface with constant Reeb curvature, then the unit normal vector field~$N$ of $M$ is singular}. In fact, since $M$ has constant Reeb function, \eqref{eq: 3.10} becomes
$$
4 \sum_{\nu=1}^{3} \en(\xi) \pn \x=0
$$
By using \eqref{eq: Reeb vector}, this equation yields $\e(\xo) \po \x =0$. From our assumption of $\e(X)\e(\xo) \neq 0$ and \eqref{eq: 4.7}, it leads to $\po \X =0$. Taking the inner product with $\po \X$, it implies
$$
g(\po \X, \po \X) = - g(\po ^{2}\X, \X) = g(\X, \X) - \big(\e_{1}(\X)\big)^{2}= 1,
$$
which gives us a contradiction. }
\end{remark}

By using Lemma~$\rm B$, in the latter part of this section we prove that the normal vector field~$N$ of $M$ is singular, when a Hopf real hypersurface~$M$ in $\GBt$ has non-vanishing geodesic Reeb flow~$\alpha=g(A\xi, \xi)$.
\begin{lemma}\label{lemma 4.3}
Let $M$ be a Hopf real hypersurface with non-vanishing geodesic Reeb flow in complex two-plane Grassmannians~$\GBt$, $m \geq 3$. If the structure Jacobi operator~$\R$ of $M$ is quadratic Killing, then the unit normal vector field~$N$ of $M$ is singular.
\end{lemma}

\begin{proof}
In \cite{LL2017}, Lee and Loo show that {\it if $M$ is Hopf, then the Reeb function~$\alpha$ is constant along the direction of structure vector field~$\xi$}, that is, $\xi \alpha =0$. Then we see that the distribution $\mathcal Q$- and the $\mathcal Q^{\bot}$-component of $\xi$ is invariant by the shape operator~$A$, that is $A\X = \alpha \X$ and $A\xo = \alpha \xo$.

\vskip 3pt

Bearing in mind of these facts, putting $X=\X$ and $Y= \xo$ in \eqref{eq: 4.4} and using \eqref{eq: 4.5}, we obtain
\begin{equation*}
\begin{split}
& \quad  - \alpha \e(\X) \p \xo  + \mu \eta(\xo) \p \X + \mu \e(\X) \p  \xo + 3\alpha (\nabla_{\X} A)\xo  + 2\alpha \eta(\xo) \phi \X \\
&  \quad -\alpha^{3} \eta(\xo) \p \X + \mu \alpha^{2} \eta(\xo) \p \X - \alpha^{3} \eta(\X) \p \xo + \mu \alpha^{2} \eta(\X) \p \xo \\
& \quad + \sum_{\nu=1}^{3} \Big [ \alpha \en(\xo) \pn \X - 3 \alpha g(\pn \X, \p \xo) \pn \xi - 2 \alpha \e(\X) \en(\xo) \pn \x  + \en(\xo) A \pn \X \\
& \quad  \quad \quad \quad   - 2\e(\X)\en(\xo) A \pn \x - 8 \alpha \e(\X) \e(\xo) \en(\xi) \pn \xi + \alpha \eta_\nu(\xo)\phi_\nu \X \Big ] =0,
\end{split}
\end{equation*}
where we have used $g(\p \xo, \X) = - g(\phi \X, \xo)=0$ and
$$
g(\pn \X, \xo) = g(\pn \x, \X) = g(\pn \x, \xo) = g(\pn\p \X, \xo) =0
$$
for all $\nu=1,2,3$. Since $\e_{2}(\xi)=\e_{3}(\x)=0$, together with $g(\po \X, \po \X) = 1$, this equation can be rearranged as
\begin{equation}\label{eq: 4.9}
\begin{split}
& - \alpha \e(\X) \p \xo  + \mu \eta(\xo) \p \X + \mu \e(\X) \p  \xo + 3\alpha (\nabla_{\X} A)\xo  \\
& + 2\alpha \eta(\xo) \phi \X -\alpha^{3} \eta(\xo) \p \X + \mu \alpha^{2} \eta(\xo) \p \X \\
& - \alpha^{3} \eta(\X) \p \xo + \mu \alpha^{2} \eta(\X) \p \xo +  2 \alpha \po \X - 5 \alpha \e(\X)  \p \xi_{1}\\
&  + \mu \po \X - 2\mu \e(\X) \po \x - 8 \alpha \e(\X) \big(\e(\xo)\big)^{2} \po \xi  =0.
\end{split}
\end{equation}
From \eqref{eq: 4.7} and \eqref{eq: 4.8}, \eqref{eq: 4.9} becomes
\begin{equation}\label{eq: 4.10}
\begin{split}
& \e^{2}(\X) \big \{ - 6\alpha - \mu  - \alpha^{3} + \mu \alpha^{2} - 8 \alpha \e^{2}(\xo) \big\}\po \X  \\
& \quad  - \e^{2}(\xo)\big \{ \mu  + 2\alpha -\alpha^{3} + \mu \alpha^{2}   \big \} \po \X  \\
& \quad + ( 2 \alpha  + \mu ) \po \X  +  3\alpha (\nabla_{\X} A)\xo =0.
\end{split}
\end{equation}

On the other hand, from \eqref{eq: 3.4} and \eqref{eq: 3.10}, the assumption $A\xo = \alpha \xo$ yields
\begin{equation*}
\begin{split}
(\nabla_{X}A)\xi_{1}& = (X \alpha) \xo + \alpha \nabla_{X}\xo - A(\nabla_{X}\xo) \\
&  = (X \alpha) \xo + \alpha \{ q_{3}(X) \xt - q_{2}(X) \xh + \po AX\} \\
& \quad \ \ - q_{3}(X) A\xt + q_{2}(X) A\xh - A\po AX \\
&  = 4\e(\xo) g(\p \xi, X) \xo + \alpha \{ q_{3}(X) \xt - q_{2}(X) \xh + \po AX\} \\
& \quad \ \ - q_{3}(X) A\xt + q_{2}(X) A\xh - A\po AX
\end{split}
\end{equation*}
for any tangent vector field~$X$ on $M$. From this, taking the inner product with $\po \X$ to $\eqref{eq: 4.10}$ and \eqref{eq: 3.4}, together with $\alpha \mu = \alpha^{2}+ 4 \eta^{2}(\X)$, we get
\begin{equation}\label{eq: 4.11}
\begin{split}
& \e^{2}(\X) \big \{ - 14 \alpha - \mu  + 12 \alpha \e^{2}(\X)\big\}   - \e^{2}(\xo)\big \{ \mu  + 2\alpha + 4 \alpha \e^{2}(\X) \big \} \\
& \quad +  2 \alpha  + \mu -12 \alpha  \eta^{2}(\X)=0,
\end{split}
\end{equation}
where we have used $g(\po \X, \po \X) = 1$, $\eta^{2}(\X) + \eta^{2}(\xo)=1$, and
\begin{equation*}
\begin{split}
g((\nabla_{\X}A)\xi_{1}, \po \X)  & = \alpha g(\po A\X, \p_{1}\X) - g(A\po A\X, \p_{1}\X) \\
&  = \alpha^{2} - \alpha \mu  = -4 \eta^{2}(\X).
\end{split}
\end{equation*}
By using non-vanishing Reeb function $\alpha \neq 0$ and $ \alpha \mu = \alpha^{2}+ 4 \eta^{2}(\X)$, together with $\eta^{2}(\xo) = 1 - \eta^{2}(\X)$, \eqref{eq: 4.11} becomes
\begin{equation}\label{eq: 4.12}
\begin{split}
& \e^{2}(\X) \big \{ - 15 \alpha^{2} - 4 \eta^{2}(\X) + 12 \alpha^{2}\eta^{2}(\X)\big\}   \\
& \ \ - \e^{2}(\xo)\big \{ 3\alpha^{2} + 4 \eta^{2}(\X) + 4 \alpha^{2} \e^{2}(\X) \big \} +  3 \alpha^{2} +4\eta^{2}(\X) -12 \alpha^{2}  \eta^{2}(\X)\\
& =  \e^{2}(\X) \big \{ - 15 \alpha^{2} - 4 \eta^{2}(\X) + 12 \alpha^{2}\eta^{2}(\X)\big\} - 4 \alpha^{2} \e^{2}(\X)\\
& \quad  + \e^{2}(\X)\big \{ 3\alpha^{2} + 4 \eta^{2}(\X) + 4 \alpha^{2} \e^{2}(\X) \big \} -12 \alpha^{2}  \eta^{2}(\X)\\
& = \e^{2}(\X) \big \{ - 12 \alpha^{2} + 16 \alpha^{2}\eta^{2}(\X) \big\} - 16 \alpha^{2}\eta^{2}(\X) \   \\
&= \e^{2}(\X) \big \{ - 28 \alpha^{2} + 16 \alpha^{2}\eta^{2}(\X) \big\} =0.
\end{split}
\end{equation}
By virtue of $\xi = \e(\X) \X + \e(\xo)\xo$ in \eqref{eq: Reeb vector} for $\e(\X)\e(\xo) \neq 0$, and our assumption of non-vanishing geodesic Reeb flow, that is, $\alpha \neq 0$, \eqref{eq: 4.12} implies that $\e^{2} (\X) = \frac{7}{4}$. Since the structure vector field~$\xi$ is unit, we should have $\e^{2}(\X)+ \e^{2}(\xo) =1$. From these facts, we obtain $\eta^{2}(\xo) =~- \frac{3}{4}$. It makes a contradiction. This means that either $\xi=\e(\X)\X = \pm \X \in \Q$ or $\xi=\e(\xo)\xo = \pm \xo \in \QP$, which gives the unit normal tangent $N$ is singular.
\end{proof}

Summing up Lemmas~$\rm B$ and \ref{lemma 4.3}, we assert that our Theorem 1 in the introduction.

\vskip 17pt

\section{quadratic Killing structure Jacobi operator for $JN \in \mathcal J N$}\label{section 5}
\setcounter{equation}{0}
\renewcommand{\theequation}{5.\arabic{equation}}

Hereafter, let $M$ be a Hopf real hypersurface with quadratic Killing structure Jacobi operator in complex two-plane Grassmannians~$\GBt$ for $m \geq 3$. Then by Theorem~1 our discussions can be divided into two cases according as the Reeb vector field $\xi \in \QP$ or $ \xi \in \Q$.

\vskip 3pt

In this section, we consider the case of $\xi \in \QP$ (i.e. $JN \in \mathcal J N$ where $N$ is a unit normal vector field on $M$ in $\GBt$, $m \geq 3$). Since $\QP$ is 3-dimensional distribution defined by $\QP = \mathrm{span}\{\xo, \xt, \xh \}$, we may put $\xi = \xo$. From this, we give an important lemma as follows.
\begin{lemma}\label{lem 5.1}
Let $M$ be a real hypersurface in complex two-plane Grassmannians $\GBt$, $m \geq 3$. Let $J_{1} \in \mathcal J$ be the almost Hermitian structure such that $JN = J_{1}N$ (or $\xi = \xo)$. Then we obtain
\begin{equation*}
\phi AX = 2 g(AX, \xh)\xt - 2g(AX, \xt)\xh + \phi_{1}AX
\end{equation*}
for any tangent vector field~$X$ on $M$.
\end{lemma}

\begin{proof}
Differentiating $\xi = \xo$ along any vector field~$X \in TM$ and using \eqref{eq: 3.4}, we obtain
\begin{equation}\label{e: 5.1}
\begin{split}
\p AX & = \nabla_{X}\xi \\
&  = \nabla_{X}\xo = q_{3}(X) \xt - q_{2}(X) \xh + \po AX.
\end{split}
\end{equation}
Taking the inner product of \eqref{e: 5.1} with $\xt$ and $\xh$, we obtain
\begin{equation*}
g(\p AX, \xt) = q_{3}(X) + g(\phi_{1} A\xi, \xt)
\end{equation*}
and
\begin{equation*}
g(\p AX, \xh) = -q_{2}(X) + g(\phi_{1} A\xi, \xh)
\end{equation*}
respectively. It follows that
\begin{equation*}
q_{3}(X) = 2 g(AX, \xh)  \quad \mathrm{and} \quad q_{2}(X) = 2 g(AX, \xt).
\end{equation*}
From this, \eqref{e: 5.1} becomes
\begin{equation}\label{e: 5.2}
\phi AX = 2 g(AX, \xh)\xt - 2g(AX, \xt)\xh + \phi_{1}AX
\end{equation}
for any tangent vector field~$X$ on $M$. Moreover, taking the symmetric part of \eqref{e: 5.2} we obtain
\begin{equation}\label{e: 5.3}
A \phi X = 2 \eta_{3}(X) A \xt - 2\eta_{2}(X) A \xh + A \phi_{1} X.
\end{equation}
\end{proof}

Then, by virtue of Lemma~\ref{lem 5.1}, we prove the following
\begin{lemma}\label{lem 5.2}
Let $M$ be a Hopf hypersurface with quadratic Killing structure Jacobi operator in complex complex two-plane Grassmannians $\GBt$, $m\geq3$. If the Reeb vector field $\xi$ belongs to $\QP$ (i.e. $\x =\xo$), then the distribution $\QP$ is invariant by the shape operator~$A$ of $M$, that is, $g(A\Q, \QP)=0$.
\end{lemma}

\begin{proof}
By \eqref{eq: 3.10} we obtain $X \alpha = (\xi \alpha)\eta(X)$ for any $X \in TM$, when the Reeb vector field $\xi$ belongs to the distribution $\Q$. From this and taking the inner product of \eqref{eq: 4.4} with $\xi$, we have
\begin{equation*}
\begin{split}
& -g(\p AX, Y) + g(A \p X, Y) + (\xi \alpha)  g(AX, Y)  - \alpha (\xi \alpha) \eta(X) \e(Y)+ 3 \alpha^{2} g(\phi AX, Y)  \\
& - \alpha g(A\phi AX, Y) + 3\alpha g(\phi X, Y) +\alpha^{2} g(A \p X, Y)-\alpha^{2} g(\p AX,Y)  \\
& + \sum_{\nu=1}^{3} \Big [ -\en(\xi) g(\pn AX, Y)   -  g(\pn \x, AX)\en(Y) - 3 g(AX,\xn)g(\pn \x, Y)\\
& \quad \quad \quad \  + 4 \alpha \en(\x) \eta(X) g(\pn \x, Y)+ \en(\xi) g(A \pn X, Y)  - \en(X) g(\pn \xi,  AY) \\
& \quad \quad \quad \   - 3g(\pn \x, X)g(\xn,  AY)+ 4 \alpha \en(\xi) g(\pn \x, X) \e(Y)  \\
& \quad \quad \quad \ - 9 \alpha g(\pn \xi,  X)\en(Y)   - 3 \alpha \en(X) g(\pn \xi, Y )+ 3 \alpha \en(\x) g(\phi_\nu X, Y)  \Big ] =0,
\end{split}
\end{equation*}
where we have used
\begin{equation*}
\begin{split}
g((\nabla_{X}A)Y, \xi)  = g((\nabla_{X}A)\xi, Y) &= (X \alpha) \eta(Y) + \alpha g(\phi A X, Y) - g(A \phi AX, Y)\\
& = (\xi \alpha) \eta(X) \eta(Y) + \alpha g(\phi A X, Y) - g(A \phi AX, Y),
\end{split}
\end{equation*}
\begin{equation*}
g(\pn \p AX, \x) = g(\p \pn \x, AX) = g(\p^{2} \xn, AX) = -g(\xn, AX) + \alpha \eta(\xn) \eta(X),
\end{equation*}
and
\begin{equation*}
g(\pn \p X, \x) = g(\p^{2}\xn, X) = - \en(X) + \en(\xi) \e(X)
\end{equation*}
for any tangent vector fields $X$ and $Y$ on $M$.

\vskip 3pt

On the other hand, from the assumption $\xi = \xo \in \QP$ we get $\p_{2} \x = \p_{2} \xo = -\xh$ and $\p_{3} \x = \p_{3} \xo = \xt$. By using these formulas into the preceding equation, we get
\begin{equation}\label{e: 5.4}
\begin{split}
& -g(\p AX, Y) + g(A \p X, Y) + (\xi \alpha)  g(AX, Y)  - \alpha (\xi \alpha) \eta(X) \e(Y)\\
& + 2 \alpha^{2} g(\phi AX, Y) - \alpha g(A\phi AX, Y) + 3\alpha g(\phi X, Y) +\alpha^{2} g(A \p X, Y)  \\
& -g(\po AX, Y)  -  2 \e_{3}(AX) \e_{2}(Y) + 2 \eta_{2}(AX) \e_{3}(Y) \\
& + g(A\po X, Y) - 2 \e_{2}(X) g(A\xh, Y) + 2 \e_{3}(X) g(A\xt, Y) \\
& + 6\alpha \e_{3}(X)\e_{2}(Y) - 6\alpha \e_{2}(X) \e_{3}(Y) + 3 \alpha g(\phi_{1}X, Y) =0.
\end{split}
\end{equation}
Deleting $Y$ from \eqref{e: 5.4}, we get
\begin{equation}\label{e: 5.5}
\begin{split}
&- \p AX + A \p X + (\xi \alpha) AX - \alpha (\xi \alpha) \eta(X) \xi + 2 \alpha^{2} \phi AX - \alpha A\phi AX  +\alpha^{2} A \p X \\
& -\po AX -  2 \e_{3}(AX) \xt + 2 \eta_{2}(AX) \x_{3}  + A\po X - 2 \e_{2}(X) A\xh + 2 \e_{3}(X) A\xt \\
& + 3\alpha \big \{ 2\e_{3}(X)\x_{2} - 2\e_{2}(X) \x_{3} + \phi X + \phi_{1}X \big \}=0
\end{split}
\end{equation}
for any tangent vector field $X$ on $M$.

\vskip 3pt

On the other hand, when $\xi=\xo \in \Q$, \eqref{eq: 3.11} gives us
\begin{equation}\label{e: 5.6}
\p X + \po X - 2 \e_{2}(X)\xh + 2 \e_{3}(X)\xt =  A\phi AX - \frac{\alpha}{2} A \phi X - \frac{\alpha}{2} \phi AX
\end{equation}
for any tangent vector field~$X$ on $M$. Substituting \eqref{e: 5.6} into \eqref{e: 5.5}, it follows that
\begin{equation*}
\begin{split}
& -\p AX + A \p X + (\xi \alpha) AX - \alpha (\xi \alpha) \eta(X) \xi + 2 \alpha^{2} \phi AX - \alpha A\phi AX  +\alpha^{2} A \p X  \\
& -\po AX -  2 \e_{3}(AX) \xt + 2 \eta_{2}(AX) \x_{3}  + A\po X - 2 \e_{2}(X) A\xh + 2 \e_{3}(X) A\xt \\
& + 3\alpha \big \{ A\phi AX - \frac{\alpha}{2} A \phi X - \frac{\alpha}{2} \phi AX \big \}=0,
\end{split}
\end{equation*}
which implies
\begin{equation}\label{e: 5.7}
\begin{split}
& (-2 + 7\alpha^{2}) \p AX + (2- \alpha^{2}) A \p X + 2(\xi \alpha) AX - 2\alpha (\xi \alpha) \eta(X) \xi   \\
& \quad  + 4 \alpha A\phi AX  - 2 \big \{ \po AX + 2 \e_{3}(AX) \xt -2 \eta_{2}(AX) \x_{3}  \big \}  \\
& \quad + 2 \big \{ A\po X  - 2 \e_{2}(X) A\xh +2  \e_{3}(X) A\xt \big \} =0
\end{split}
\end{equation}
for any $X \in TM$. Bearing in mind of \eqref{e: 5.2} and \eqref{e: 5.3}, the above equation reduces to
\begin{equation}\label{e: 5.8}
\begin{split}
&(-4 + 7\alpha^{2}) \p AX + (4- \alpha^{2}) A \p X + 2(\xi \alpha) AX \\
& \quad - 2\alpha (\xi \alpha) \eta(X) \xi + 4 \alpha A\phi AX  =0.
\end{split}
\end{equation}
From \eqref{e: 5.2} and \eqref{e: 5.3}, we get
\begin{equation}\label{e: 5.9}
2 \eta_{3}(AX) \xt - 2 \eta_{2}(AX) = \phi AX - \phi_{1}AX
\end{equation}
and
\begin{equation}\label{e: 5.10}
2 \eta_{3}(X) A \xt - 2 \eta_{2}(X) A\xh = A \phi X - A \phi_{1}X,
\end{equation}
respectively. Substituting \eqref{e: 5.9} and \eqref{e: 5.10} into \eqref{e: 5.7}, it becomes
\begin{equation*}
\begin{split}
& (-2 + 7\alpha^{2}) \p AX + (2- \alpha^{2}) A \p X + 2(\xi \alpha) AX - 2\alpha (\xi \alpha) \eta(X) \xi + 4 \alpha A\phi AX   \\
& - 2 \big \{ \po AX + \phi AX - \phi_{1}AX \big \}  + 2 \big \{ A\po X  - A \phi X + A \phi_{1}X  \big \} =0,
\end{split}
\end{equation*}
which yields
\begin{equation}\label{e: 5.11}
\begin{split}
& (-4 + 7\alpha^{2}) \p AX - \alpha^{2} A \p X + 2(\xi \alpha) AX - 2\alpha (\xi \alpha) \eta(X) \xi \\
& \quad + 4 \alpha A\phi AX   + 4 A \po X =0.
\end{split}
\end{equation}
Subtracting \eqref{e: 5.11} from \eqref{e: 5.8}, we have $A \phi X = A \po X$, which means that $\phi A X = \po AX$ for any tangent vector field~$X$ of $M$. From this, \eqref{e: 5.2} becomes
\begin{equation}\label{e: 5.12}
g(A\xt, X) \xt - g(A\xt, X) \xh = 0
\end{equation}
for any tangent vector field $X$ of $M$. Taking the inner product of \eqref{e: 5.12} with $\xt$ (resp. $\xh$), we get the following for any tangent vector field $X$ of $M$
\begin{equation}
g(A\xt, X) = g(AX, \xt) = 0 \quad (\mathrm{resp.}\ \  g(A\xh, X) = g(AX, \xh) =0),
\end{equation}
which means that $g(A\Q, \QP)=0$. It gives a complete proof of Lemma~\ref{lem 5.2}.
\end{proof}

By Theorems~$\rm A$ and Propositions~$\rm A$, Lemma~\ref{lem 5.2} assure that {\it if a Hopf real hypersurface satisfies all  geometric conditions mentioned in Lemma~\ref{lem 5.2}, then $M$ is locally congruent to an open part of the model spaces of type~$(\mathcal T_{A})$}.

\vskip 6pt

From now on, we will check whether a real hypersurface of type~$(\mathcal T_{A})$ satisfy our hypothesis given in Lemma~\ref{lem 5.2}. By Proposition~$A$ mentioned in section~\ref{section 2}, we see that such real hypersurface is Hopf and its normal vector field satisfies $JN \in \mathcal JN$.

\vskip 3pt

In the remained part of this section, we want to check if the structure Jacobi operator~$\R$ for a model space of type~$(\mathcal T_{A})$ satisfies the cyclic parallelism. In order to do this, we want to find some necessary and sufficient conditions for structure Jacobi operator~$\R$ of a real hypersurface~$(\mathcal T_{A})$ to be quadratic Killing according to each eigenspace including the vector $Y$.

\vskip 3pt

From such a view point, first, we consider the following case.
\vskip 3pt

\noindent {\bf Case A.} $Y \in T_{\lambda}$

\vskip 2pt

In other words, from \eqref{eq: 4.4} and \eqref{eq: 4.5}, together with \eqref{e: 2.2}, the structure Jacobi operator~$\R$ of a real hypersurface of type $(\mathcal T_{A})$ satisfies the following for any tangent vector field $X \in T(\mathcal T_{A})$
\begin{equation}\label{e: 5.14}
\begin{split}
&  3\alpha ( \lambda^{2}- \alpha \lambda -2  )g(\phi Y, X) \xi  -2 (2 \alpha - \beta - \lambda) g(\p_{2}Y, X) \xt \\
& \quad -2 (2 \alpha - \beta - \lambda) g(\p_{3}Y, X) \xh -2 (2 \alpha - \beta - \lambda) \eta_{2}(X) \phi_{2}Y \\
& \quad -2 (2 \alpha - \beta - \lambda) \eta_{3}(X) \phi_{3}Y  =0,
\end{split}
\end{equation}
where $T(\mathcal T_{A})$ denotes a tangent space of type~$(\mathcal T_{A})$ and we have used $\phi \phi_{2} Y = \phi_{2} \phi Y = - \phi_{3}Y \in T_{\mu}$ and $\phi \phi_{3} Y = \phi_{3} \phi Y = \phi_{2}Y \in T_{\mu}$ for any $Y \in T_{\lambda}$.

\vskip 3pt

From now on, we want to check a solution of the equation \eqref{e: 5.14} to be satisfied for type~$(\mathcal T_{A})$. In fact, the left side of \eqref{e: 5.14} depend on the eigenspaces of $(\mathcal T_{A})$ is given as
\begin{equation*}
\mathrm{Left\ Side\ of\ } \eqref{e: 5.14} = \left\{ \begin{array}{cl}
0 & \mbox{for} \ X \in T_{\alpha} \\
-2 (2 \alpha - \beta - \lambda) \phi_{2}Y & \mbox{for} \ X =\xt \in T_{\beta} \\
-2 (2 \alpha - \beta - \lambda) \phi_{3}Y  & \mbox{for} \ X=\xh \in T_{\beta} \\
3\alpha ( \lambda^{2}- \alpha \lambda -2  )g(\phi Y, X) \xi  & \mbox{for} \ X \in T_{\lambda} \\
-2 (2 \alpha - \beta - \lambda) g(\p_{2}Y, X) (\xt + \xh) & \mbox{for} \ X \in T_{\mu}, \\
\end{array}\right.
\end{equation*}
for $Y \in T_{\lambda}$. By using $\alpha = 2 \sqrt{2} \cot (2 \sqrt{2}r)=\sqrt{2}(\cot (\sqrt{2}r) - \tan(\sqrt{2}r))$ and $\lambda = - \sqrt{2} \tan (\sqrt{2}r)$ with $r \in (0, \frac{\pi}{2 \sqrt{2}})$, we get
$\lambda^{2}- \alpha \lambda -2 =0$. From this, the previous formula follows
\begin{equation}\label{e: 5.15}
\begin{split}
& \mathrm{Left\ Side\ of\ } \eqref{e: 5.14} \\
& \quad \quad = \left\{ \begin{array}{cl}
0 & \mbox{for} \ X \in T_{\alpha} \\
-2 (2 \alpha - \beta - \lambda) \phi_{2}Y & \mbox{for} \ X =\xt \in T_{\beta} \\
-2 (2 \alpha - \beta - \lambda) \phi_{3}Y  & \mbox{for} \ X=\xh \in T_{\beta} \\
0 & \mbox{for} \ X \in T_{\lambda} \\
-2 (2 \alpha - \beta - \lambda) g(\p_{2}Y, X) (\xt + \xh) & \mbox{for} \ X \in T_{\mu}, \\
\end{array}\right.
\end{split}
\end{equation}
for $Y \in T_{\lambda}$.

\vskip 3pt

Bearing in mind of Proposition~$\rm A$, if $r=\frac{\pi}{4\sqrt{2}}$, then $2 \alpha - \beta - \lambda =0$. Hence, when $Y \in T_{\lambda}$, the structure Jacobi operator $\R$ is quadratic Killing if and only the radius~$r$ of the tube $(\mathcal T_{A})$ is $\frac{\pi}{4\sqrt{2}}$. It means $\alpha = \mu = 0$, $\beta = \sqrt{2}$, and $\lambda= - \sqrt{2}$.

\vskip 6pt

\noindent {\bf Case B.} $Y \in T_{\alpha} \oplus T_{\beta} \oplus T_{\mu}$

\vskip 2pt

Under these situations, we consider our problem for the other cases, that is, $Y \in T_{\alpha} \oplus T_{\beta} \oplus T_{\mu}$. Since $\alpha =0$, the left side of \eqref{eq: 4.4} becomes
\begin{equation}\label{e: 5.16}
\begin{split}
& \mathrm{Left\ Side\ of\ } \eqref{eq: 4.4} \\
& = -g(\p AX, Y) \xi - \e(Y) \p AX -g(\p AY, X) \xi \\
& \quad - \e(X) \p AY  + \eta(Y)A\p X + \e(X) A \p  Y \\
\end{split}
\end{equation}
\begin{equation*}
\begin{split}
& \quad   - \sum_{\nu=1}^{3} \Big [ g(\pn AX, Y) \xn + 2\e(Y) g(\pn \x, AX) \xn  + \en(Y) \pn AX \\
&  \quad \quad \quad \ \   +3 g(\pn AX, \p Y) \pn \xi + 3 \e(Y) \en(AX) \pn \x   \\
& \quad \quad \quad  \ \  - 3g(\pn \x, Y) \pn \p AX - 2g(\pn \x, AX) \pn\p Y  + \en(X) \pn AY  \\
&  \quad \quad \quad  \ \  - 2g(\pn \x, AY) \pn\p X + g(\pn AY, X) \xn + 2\e(X) g(\pn \x, AY) \xn  \\
& \quad \quad \quad   \ \   +3 g(\pn AY, \p X) \pn \xi + 3 \e(X) \en(AY) \pn \x - 3g(\pn \x, X) \pn \p AY \\
&  \quad \quad \quad  \ \    -\en(Y) A \pn X  + 2\e(X)\en(Y) A \pn \x - \en(X) A \pn Y \\
& \quad \quad \quad  \ \   - 3g(\pn \xi, Y) A \pn \p X  + 3 \e(X)  g(\pn \x, Y) A\xn \\
& \quad \quad \quad \ \ - 3g(\pn \x, X) A \p \pn Y  - 2 g(\pn\p X, Y) A \pn \x \Big ]
\end{split}
\end{equation*}
for any $X \in T(\mathcal T_{A})$ and $Y \in T_{\alpha} \oplus T_{\beta} \oplus T_{\mu}$.

\vskip 3pt

\noindent {\bf Subcase 1.} $Y = \x \in T_{\alpha}$

\vskip 2pt

Then, by using $\alpha = 0$ the left side of \eqref{e: 5.16} becomes
\begin{equation}\label{e: 5.17}
\begin{split}
&  - \p AX  +  A\p X   - \sum_{\nu=1}^{3} \big \{ g(A\pn  \xi, X) \xn  + \en(\xi) \pn AX + 3 g(A\xn, X) \pn \x \big \} \\
& \quad  \  + \sum_{\nu=1}^{3} \big \{ \en(X) A \pn \xi - 3g(\pn \x, X) A \xn  + \en(\xi) A \pn X  \big \}\\
& = - \p AX  +  A\p X - \po AX + A \phi_{1} X,
\end{split}
\end{equation}
where we have used $\phi_{2}\xi = - \xh$, $\phi_{3}\x = \xt$, and $\p \pn \xi = \p^{2}\xn = - \xn + \e(\xn)\xi$. According to the composition of the eigenspaces for $(\mathcal T_{A})$, we see that each eigenspace $T_{\sigma}$ of $(\mathcal T_{A})$ is $\phi$-(or $\po$-)invariant, that is, $\phi T_{\sigma} = \po T_{\sigma} = T_{\sigma}$. From this, \eqref{e: 5.17} vanishes on all eigenspaces of $(\mathcal T_{A})$. So, this means that the structure Jacobi operator~$\R$ is quadratic Killing when $Y \in T_{\alpha}$.

\vskip 6pt

\noindent {\bf Subcase 2.} $Y \in T_{\beta}$

\vskip 2pt

Since $T_{\beta}=\mathrm{span}\{\xt, \xh\}$, we have the following two subcases.

\vskip 3pt

\begin{itemize}
\item {$Y =\xt \in T_{\beta}$ \\
Using $\alpha = 0$ and \eqref{e: 5.16} can be rearranged as
\begin{equation}\label{e: 5.18}
\begin{split}
& 6 \beta \e_{3}(X) \xi + \beta\eta(X) \xh - \p_{2} AX + 3 \p_{3} \p AX \\
& \ \ + 2 \beta \p_{3}\p X + A\p_{2} X + 3 A \p_{3} \p X,
\end{split}
\end{equation}
for any eigenvector $X$ on $(\mathcal T_{A})$. It is well-known that for $X \in T_{\lambda}$ (resp. $X \in T_{\mu}$), by the straightforward calculation with \eqref{eq: 3.2}, we obtain
\begin{gather*}
\ \phi_{2}\p X \underset{X \in T_{\lambda}}= \phi_{2}\phi_{1}X \underset{\eqref{eq: 3.2}}{=} -\p_{3}X \in T_{\mu} \\
(\mathrm{resp.}\ \phi_{2}\p X \underset{X \in T_{\mu}}= -\phi_{2}\phi_{1}X \underset{\eqref{eq: 3.2}}{=} \p_{3}X \in T_{\lambda}), \\
\ \phi_{3}\p X \underset{X \in T_{\lambda}}= \phi_{3}\phi_{1}X \underset{\eqref{eq: 3.2}}{=} \p_{2}X \in T_{\mu} \\
\end{gather*}
\begin{gather*}
(\mathrm{resp.} \ \phi_{3}\p X \underset{X \in T_{\mu}}= -\phi_{3}\phi_{1}X \underset{\eqref{eq: 3.2}}{=} -\p_{2}X \in T_{\lambda} ),
\end{gather*}
and
\begin{gather*}
\phi X = \po X \in T_{\lambda} \ \ (\mathrm{resp.}\ \phi X = \po X \in T_{\mu}).
\end{gather*}
Bearing in mind such properties, together with $\beta=\sqrt{2}$ and $\lambda = - \sqrt{2}$, \eqref{e: 5.18} is identically vanishing for any tangent vector field $X$ on $(\mathcal T_{A})$. }
\item {$Y =\xh \in T_{\beta}$ \\
Similarly, from \eqref{e: 5.16} we obtain
\begin{equation}\label{e: 5.19}
\begin{split}
&-6 \beta \e_{2}(X) \xi - \beta \eta(X) \xt - \p_{3} AX  - 3 \p_{2} \p AX \\
& \quad - 2 \beta \p_{2}\p X + A\p_{3} X - 3 A \p_{2} \p X,
\end{split}
\end{equation}
for any eigenvector $X$ on $(\mathcal T_{A})$. More specifically, according to each eigenspace $T_{\alpha}$, $T_{\beta}$, $T_{\lambda}$ and $T_{\mu}$, it follows
\begin{equation*}
\eqref{e: 5.19} = \left\{ \begin{array}{ll}
- \beta \xt + A \phi_{3}\x = - \beta \xt + A \xt = 0 & \mbox{for} \ X \in T_{\alpha} \\
-6 \beta  \xi - \phi_{3} A \xt - 3 \phi_{2} \phi A \xt - 2 \beta \phi_{2} \phi \xt =0 & \mbox{for} \ X =\xt \in T_{\beta} \\
-3 \phi_{2}\phi A \xh - 2 \beta \phi_{2} \phi \x_{3} - 3 A \phi_{2} \phi \xh = 0  & \mbox{for} \ X=\xh \in T_{\beta} \\
- \lambda \phi_{3}X - 3\lambda \phi_{2}\phi X - 2 \beta \phi_{2} \phi X = 2 (\lambda + \beta)\phi_{3}X=0 & \mbox{for} \ X \in T_{\lambda} \\
- 2 \beta \p_{2} \p X + \lambda \phi_{3}X - 3 A \phi_{3}X =  -2\beta (\beta + \lambda) \phi_{3}X =0& \mbox{for} \ X \in T_{\mu} \\
\end{array}\right.
\end{equation*}
where we have used $\phi_{2}\phi \xt = - \phi_{2}\xh = -\x$, $\phi_{2}\phi \xt = \phi_{2}\xt =0$, $\beta= \sqrt{2}$ and $\lambda=-\sqrt{2}$. }
\end{itemize}

\vskip 6pt

\noindent {\bf Subcase 3.} $Y \in T_{\mu}$

\vskip 2pt

Since $Y \in T_{\mu}$, we see that $\mu=0$ and $\phi X = \phi_{1} X \in T_{\mu}$. From these properties, \eqref{e: 5.16} becomes
\begin{equation}\label{e: 5.20}
\begin{split}
& - \sum_{\nu=1}^{3} \big \{ g(\phi_{\nu}AX, Y) \xn + 3 g(\p_{\nu} AX, \p Y) \pn \x - 2g(\pn \x, AX) \pn \p Y \\
& \quad \quad \quad - \en(X) A \pn Y - 3g(\pn \x, X) A \p \pn Y - 2 g(\pn \p X, Y) A \pn \xi \big \} \\
& = -2 (\beta + \lambda) \big \{ g(\p_{2}X, Y) \xt + g(\p_{3}X, Y) \xh + \e_{2}(X) \p_{2}Y  + \e_{3}(X) \p_{3}Y \big\},
\end{split}
\end{equation}
where we have used $\p_{2} \p Y = \p_{3}Y \in T_{\lambda}$ and $\p_{3} \p Y = -\p_{2}Y \in T_{\lambda}$ for any eigenvector~$X$ on $(\mathcal T_{A})$ and $Y \in T_{\mu}$. Since $\beta= \sqrt{2}$ and $\lambda=-\sqrt{2}$, \eqref{e: 5.20} is identically vanishing for any tangent vector field~$X$ on
$(\mathcal T_{A})$.

\vskip 6pt

Summing up these discussions, we assert that {\it the structure Jacobi operator~$\R$ of a real hypersurface of type $(\mathcal T_{A})$ is quadratic Killing if and only if the radius $r$ of the tube around of type~$(\mathcal T_{A})$ is $\frac{\pi}{4\sqrt{2}}$}.

\vskip 17pt

\section{Quadratic Killing structure Jacobi operator for $JN \bot \mathcal J N$}\label{section 6}
\setcounter{equation}{0}
\renewcommand{\theequation}{6.\arabic{equation}}

Let $M$ be a Hopf real hypersurface with quadratic Killing structure Jacobi operator~$\R$ in complex two-plane Grassmannians $\GBt$, $m \geq 3$. Assume that the unit normal vector field~$N$ of $M$ satisfies $JN \bot {\mathcal J}N$ (i.e. $\xi \in \Q$). Related to the Reeb vector field $\xi$ of $M$ in $\GBt$, Lee and Suh gave:
\begin{thm B}[\cite{LS}]
Let $M$ be a connected orientable Hopf real hypersurface in complex two-plane Grassmannians of compact type $\GBt$, $m \geq 3$. Then the Reeb vector~$\xi$ belongs to the distribution $\mathcal Q$ if and only if $M$ is locally congruent to an open part of $(\mathcal T_{B})$: a tube around a totally geodesic $\mathbb H P^{n}$ in $\GBt$, where $m = 2n$.
\end{thm B}
\noindent By virtue of Theorem~$1$ and Theorem~$\rm B$, we assert that {\it a Hopf real hypersurface $M$ in complex two-plane Grassmannians $\GBt$, $m \geq 3$, satisfying the hypothesis in our Theorem~$2$ is locally congruent to an open part of the model space mentioned in Theorem~$\rm B$}. Hereafter, conversely, let us check whether the structure Jacobi operator~$\R$ of the model space of type~$(\mathcal T_{B})$ satisfies our assumption of Killing structure Jacobi operator.

\vskip 6pt

In order to do this, we introduce a proposition given in~\cite{Suh2006} as follows:
\begin{pro B}\label{Proposition B}
Let $M$ be a connected real hypersurface in complex two-plane Grassmannians $\GBt$. Suppose that $A\Q \subset \Q$, $A\xi = \alpha \xi$, and $\xi$ is tangent to $\Q$. Then the quaternionic dimension $m$ of $\GBt$ is even, say $m =2n$, and $M$ has five distinct constant principal curvatures
\begin{equation*}
\alpha = -2 \tan(2 r), \  \beta= 2 \cot (2r), \  \gamma=0, \   \lambda = \cot (r), \  \mu = - \tan (r),
\end{equation*}
with some $r \in (0, \frac{\pi}{4})$. The corresponding multiplicities are
\begin{equation*}
m(\alpha) =1, \ \  m(\beta)=3=m(\gamma) \ \  m(\lambda) = 4n-4 = m(\mu)
\end{equation*}
and the corresponding eigenspaces are
\begin{equation*}
\begin{split}
&T_{\alpha} = \mathbb R \xi = \mathcal C^{\bot} = \mathrm{span}\{\xi\}, \\
&T_{\beta} = \mathcal J J\xi = \mathrm{span}\{\xo, \xt, \xh\}, \\
&T_{\gamma}  = \mathcal J \xi = \mathrm{span}\{\p \xo, \p \xt, \p \xh\},\\
&T_{\lambda}, \ \  T_{\mu},
\end{split}
\end{equation*}
where
$$
T_{\lambda} \oplus T_{\mu} = TM \ominus (\mathbb R \xi \oplus \mathcal J  J \xi ), \ \  \mathcal J T_{\lambda}=  T_{\lambda}, \ \  \mathcal J T_{\mu} = T_{\mu},\ \   J T_{\lambda} =  T_{\mu}.
$$
\end{pro B}

\noindent In order to check the converse part, we assume that the structure Jacobi operator~$\R$ of our model space of type~$(\mathcal T_{B})$ satisfies the property of quadratic Killing. Accordingly, by $A \phi \x_{\nu} =0$ for $\nu = 1,2,3$, the property~\eqref{eq: Killing structure Ja op} can be rearranged as
\begin{equation}\label{eq: 6.1}
\begin{split}
& g(X, A\p Y) \xi - \e(Y) \p AX -g(X, \p AY) \xi - \e(X) \p AY  + \eta(Y)A\p X \\
& + \e(X) A \p  Y  + \alpha^{2} \e(Y) A \phi  X  + \alpha^{2} \eta(X) A \phi Y + 3\alpha (\nabla_X A)Y   \\
& + \alpha g(\phi X, Y)\xi + 3\alpha \eta(Y) \phi X +\alpha^{2}g(X,A \p Y)\xi -\alpha^{2} g(\p AY,X)\xi \\
& - 2 \alpha g(\phi Y,X)\xi  -\alpha ^{2}\eta(Y) \p AX - \alpha^{2} \eta(X) \p AY \\
&   + \sum_{\nu=1}^{3} \Big [ -g(\pn AX, Y) \xn  - \en(Y) \pn AX -3 g(\pn AX, \p Y) \pn \xi \\
& \quad \quad \quad  \ - 3 \e(Y) \en(AX) \pn \x  + 3g(\pn \x, Y) \pn \p AX  \\
& \quad \quad \quad  \ - 3 \alpha \e (X) g(\pn \xi, Y) \xn + 2g(\pn \x, AX) \pn\p Y \\
\end{split}
\end{equation}
\begin{equation*}
\begin{split}
& \quad \quad \quad  \ - g(\pn AY, X) \xn - 2\e(X) g(\pn \x, AY) \xn  - \en(X) \pn AY \\
& \quad \quad \quad   \ -3 g(\pn AY, \p X) \pn \xi - 3 \e(X) \en(AY) \pn \x + 3g(\pn \x, X) \pn \p AY  \\
& \quad \quad \quad   \ - 3 \alpha \e (Y) g(\pn \xi, X) \xn + 2g(\pn \x, AY) \pn\p X +\en(Y) A \pn X  \\
& \quad \quad \quad  \ + \en(X) A \pn Y  + 3g(\pn \xi, Y) A \pn \p X - 3 \e(X)  g(\pn \x, Y) A\xn\\
& \quad \quad \quad   \ + 3g(\pn \x, X) A \p \pn Y  - 3 \alpha g(\pn \xi, X) \en(Y)\xi + \alpha g(\phi_\nu X, Y) \xn  \\
& \quad \quad \quad  \  + 2\alpha \en(Y) \phi_\nu X - \alpha g(\phi_\nu\phi X, Y) \pn \xi + \alpha g(\pn \xi, X) \phi \phi_\nu Y + \alpha \eta_\nu(\phi X)\en(Y) \xi\\
& \quad \quad \quad   \ +\alpha \eta(X)\en(Y) \pn \xi + \alpha \eta_\nu(Y)\phi_\nu X  - 2 \alpha g(\phi_\nu Y,X)\xi_\nu  \ + \alpha \eta_\nu(\phi Y)\phi_\nu\phi X \\
&  \quad \quad \quad \
 - \alpha \eta_\nu(\phi X)\phi_\nu\phi Y + \alpha \eta(Y)\eta_\nu(\phi X)\xn - \alpha \eta(X)\eta_\nu(\phi Y)\xi_{\nu} \Big ]=0
\end{split}
\end{equation*}
for any tangent vector field $X$ on type~$(\mathcal T_{B})$.

\vskip 3pt

Bearing in mind of our assumption, the structure Jacobi operator~$\R$ for the tube of type~$(\mathcal T_{B})$ is quadratic Killing, taking $Y \in T_{\alpha}$ in \eqref{eq: 6.1} yields
\begin{equation}\label{eq: 6.2}
\begin{split}
& - \p AX + A\p X + \alpha^{2} A \phi  X + 2\alpha^{2} \p AX - 3 \alpha  A \p AX + 3\alpha \phi X \\
& -3 \sum_{\nu=1}^{3} \Big [ \beta \en(X) \pn \x +3 \alpha g(\pn \xi, X) \xn  + \alpha \en(X) \pn \xi + \beta g(\pn \x, X) \xn \Big ]=0,
\end{split}
\end{equation}
where we have used $(\nabla_{X}A)\xi = \alpha \p AX - A \p AX$ and $\phi \phi_{\nu} \x = \p^{2}\xn = - \xn$. Furthermore, taking $X=\xi_{\mu} \in T_{\beta}$ in \eqref{eq: 6.2} follows
\begin{equation*}
\begin{split}
& - \p A \x_{\mu} + A\p \x_{\mu} + \alpha^{2} A \phi  \x_{\mu} + 2\alpha^{2} \p A \x_{\mu} - 3 \alpha  A \p A \x_{\mu} + 3\alpha \phi \x_{\mu} \\
& \quad   -3 \sum_{\nu=1}^{3} \Big [ \beta \en(\x_{\mu}) \pn \x +3 \alpha g(\pn \xi, \x_{\mu}) \xn  + \alpha \en(\x_{\mu}) \pn \xi + \beta g(\pn \x, \x_{\mu}) \xn \Big ] \\
&  = - \beta \p_{\mu}  \x + 2\alpha^{2} \beta \p_{\mu} \x + 3\alpha \phi_{\mu} \x -3\beta \p_{\mu} \x  - 3 \alpha \p_{\mu}  \xi \\
&  = - 4 \beta \p_{\mu} \x + 2\alpha^{2} \beta \p_{\mu} \x = 2 \beta ( \alpha^{2} -2) \p_{\mu} \x = 0,
\end{split}
\end{equation*}
which implies $\beta ( \alpha^{2} -2) =0$. Since $\beta = 2 \cot(2r)$ for $r \in (0,\frac{\pi}{4}) $, we obtain $\alpha^{2}=2$.

\vskip 3pt

On the other hand, taking $X \in T_{\lambda}$ in \eqref{eq: 6.2}, together with $\phi T_{\lambda}=T_{\mu}$, provides
\begin{equation*}
\begin{split}
& \quad \  - \lambda \p X + \mu \p X + \alpha^{2} \mu \phi  X + 2\alpha^{2} \lambda \p X - 3 \alpha \lambda \mu   \p X + 3\alpha \phi X \\
&  = (-\lambda + \mu + 2 \mu + 4 \lambda + 3 \alpha + 3\alpha ) \phi X \\
&  = 3(\lambda + \mu + 2\alpha)\phi X = 3 (\beta + 2 \alpha) \phi X = 0,
\end{split}
\end{equation*}
where we have used $\alpha^{2}=2$, $ \lambda \mu  = \cot r  \cdot (-\tan r) =-1 $, and $\lambda + \mu = 2 \cot (2t) = \beta$.

\vskip 3pt

Applying a method to \eqref{eq: 6.2} that done above, the left side of \eqref{eq: 6.2} according to each eigenspace of type $(\mathcal T_{\beta})$ is given as
\begin{equation*}
\mathrm{Left\  Side \  of \  } \eqref{eq: 6.2} = \left\{ \begin{array}{ll}
0 & \mbox{for} \ X \in T_{\alpha} \\
2 \beta (\alpha^{2} -2) \phi_{\mu} \x & \mbox{for} \ X =\x_{\mu} \in T_{\beta} \\
-6 (\beta + 2 \alpha) \xi_{\mu} & \mbox{for} \ X=\phi_{\mu} \xi \in T_{\gamma} \\
3 (\beta + 2 \alpha) \phi X & \mbox{for} \ X \in T_{\lambda} \\
3 (\beta + 2 \alpha) \phi X & \mbox{for} \ X \in T_{\mu}. \\
\end{array}\right.
\end{equation*}

\vskip 6pt

Now, as the other case we consider the case $Y \in T_{\lambda}$. Then, by using $JT_{\lambda}=T_{\mu}$ and $\mathcal J T_{\lambda}=T_{\lambda}$, equation \eqref{eq: 6.1} is rearranged as
\begin{equation}\label{eq: 6.3}
\begin{split}
& g(X, A\p Y) \xi -g(X, \p AY) \xi - \e(X) \p AY  + \e(X) A \p  Y + \alpha^{2} \eta(X) A \phi Y \\
& + 3\alpha (\nabla_X A)Y + \alpha g(\phi X, Y)\xi +\alpha^{2}g(X,A \p Y)\xi \\
& -\alpha^{2} g(\p AY,X)\xi - 2 \alpha g(\phi Y,X)\xi  - \alpha^{2} \eta(X) \p AY \\
& + \sum_{\nu=1}^{3} \Big [ -g(\pn AX, Y) \xn  -3 g(\pn AX, \p Y) \pn \xi - g(\pn AY, X) \xn  \\
& \quad \quad \quad  \ - \en(X) \pn AY -3 g(\pn AY, \p X) \pn \xi + 3g(\pn \x, X) \pn \p AY  \\
& \quad \quad \quad  \  + \en(X) A \pn Y + \alpha g(\phi_\nu X, Y) \xn  + 3g(\pn \x, X) A \p \pn Y  \\
& \quad \quad \quad  \   - \alpha g(\phi_\nu\phi X, Y) \pn \xi + \alpha g(\pn \xi, X) \phi \phi_\nu Y \\
& \quad \quad \quad  \  - 2 \alpha g(\phi_\nu Y,X)\xi_\nu  + \alpha g( \phi_\nu \xi, X) \phi_\nu\phi Y \Big ] \\
& =  (\mu - \lambda - 3 \alpha + \alpha^{2}\mu - \alpha^{2} \lambda) g(X, \phi Y) \xi + (\lambda + \mu + \alpha^{2} \mu - \alpha^{2} \lambda ) \eta(X) \phi Y  \\
& \quad + 3 \alpha (\nabla_{X}A) Y + \sum_{\nu=1}^{3} \Big [- 3\alpha g(\phi_{\nu}Y, X)\xi_{\nu}  + (3\mu + 3\lambda - \alpha) g(\p \p_{\nu}Y, X) \p_{\nu}\xi \Big ] \\
& \quad  + \sum_{\nu=1}^{3} (3\lambda + 3 \mu + 2 \alpha) g(\phi_{\nu}\xi, X) \p \p_{\nu} Y =0
\end{split}
\end{equation}
for any tangent vector field $X$ on type~$(\mathcal T_{B})$. Restricting $X \in T_{\alpha}$ in \eqref{eq: 6.3} provides
\begin{equation}\label{eq: 6.4}
 (\lambda + \mu + \alpha^{2} \mu - \alpha^{2} \lambda ) \phi Y + 3 \alpha (\nabla_{\xi}A) Y
=0
\end{equation}
for any $Y \in T_{\lambda}$. By the equation of Codazzi \eqref{eq: 3.8}, we get
\begin{equation*}
\begin{split}
(\nabla_{\xi} A)Y & = (\nabla_Y A)\xi + \phi Y + \sum_{\nu=1}^3 \big\{- \eta_\nu(Y)\phi_\nu \xi - 3g(\phi_\nu \xi,Y)\xi_\nu  \big\} \\
&  = \alpha \phi AY - A \phi AY + \phi Y = (\alpha \lambda - \lambda \mu + 1) \phi Y
\end{split}
\end{equation*}
for any $Y \in T_{\lambda}$. From this, \eqref{eq: 6.4} becomes
\begin{equation*}
 (\lambda + \mu + \alpha^{2} \mu - \alpha^{2} \lambda + 3 \alpha^{2} \lambda - 3 \alpha \lambda \mu + 3 \alpha) \phi Y
=0.
\end{equation*}
Since $\alpha^{2}=2$, $\beta+2\alpha=0$, $\lambda + \mu =\beta$ and $\lambda \mu = -1$, the previous equation gives
\begin{equation}
\begin{split}
\beta + 2 \mu + 4 \lambda + 6 \alpha = -2 (\beta - \mu -2 \lambda)=0,
\end{split}
\end{equation}
which gives us a contradiction. In fact, by Proposition~$\rm B$ we see that $\beta=2 \cot (2r)$, $\lambda=\cot(r)$ and $\mu=-\tan(r)$ where $r \in (0, \frac{\pi}{4})$. From this, we get
$$
\beta - \mu -2 \lambda = - \frac{1}{\tan r},
$$
which means that the function $\beta - \mu -2 \lambda$ is non-vanishing for any $r \in (0, \frac{\pi}{4})$.

\vskip 6pt

Summing up those documents in this section, we can assert that {\it there does not exist a Hopf real hypersurface in complex two-plane Grassmannians $\GBt$, $m \geq 3$, with quadratic Killing structure Jacobi operator when the normal vector field of $M$ is of type $JN \bot \mathcal J N$}.


\vskip 17pt


\begin{thebibliography}{99}

\bibitem{B01} J. Berndt, {\it Riemannian geometry of compelx two-plane Grassmannians}, Rend. Sem. Mat. Univ. Politec. Torino~{\bf 55} (1997), 19-83.
%
\bibitem{BS1} J. Berndt and Y.J. Suh, {\it Real hypersurfaces in complex two-plane Grassmannians}, Monatsh. Math.~{\bf 127} (1999), 1-14.
%
\bibitem{BS2} J. Berndt and Y. J. Suh, {\it Isometric flows on real hypersurfaces in complex two-plane Grassmannians}, Monatsh. Math.~{\bf 137} (2002), 87-98.
%
\bibitem{BS2020} J. Berndt and Y.J. Suh,  {\it Real hypersurfaces in Hermitian Symmetric Spaces}, Advances in Analysis and Geometry, Editor in Chief, Jie Xiao, $\copyright$2020 Copyright-Text, Walter de Gruyter GmbH, Berlin/Boston (in Press).
%
%
\bibitem{HMS} K. Heil, A. Moroianu, and U. Semmelmann, {\it Killing and conformal Killing tensors}, J. Geom. Phys.~{\bf 106} (2016), 383-400.
%
\bibitem{JCPS2011} I. Jeong, C.J.G. Machado, J.D. P\'{e}rez, and Y.J. Suh, {\it Real hypersrufaces in complex two-plane Grassmannians with $\mathfrak D^{\bot}$-parallel strutue Jacobi operator}, Interant. J. Math.~{\bf 22}, no. 5, 655-673.
%
%
\bibitem{KoNo} S. Kobayashi and K. Nomizu, {\it Foundations of Differential Geometry, Vol. I}, Reprint of the 1963 original, Wiley Classics Library, A Wiley-Interscience Publication, John Wiley \& Sons, Inc., New York, 1996.
%
\bibitem{LL2017} R.-H. Lee and T.-H. Loo, {\it Hopf hypersurfaces in complex Grassmannians of rank two}, Results Math.~{\bf 71} (2017), 1083-1107.
%
\bibitem{LS} H. Lee, Y.J. Suh, {\it Real hypersurfaces of type~$(B)$ in complex two-plane Grassmannians related to the Reeb vector}, Bull. Korean Math. Soc.~{\bf 47} (2010), no. 3, 551-561.
%
\bibitem{LS2013} H. Lee and Y.J. Suh, {\it Remarks on the components of the Reeb vector field for real hypersurfaces in complex two-plane Grassmannians}, Proceedings of the 17th International Workshop on Differential Geometry and the 7th KNUGRG-OCAMI Differential Geometry Workshop, Vol. 17, 153–169, Natl. Inst. Math. Sci.(NIMS), Taej{\u{o}}n, 2013.
%
\bibitem{LS2017} H. Lee and Y.J. Suh, {\it Reeb recurrent structure Jacobi operator on real hypersurfaces in complex two-plane Grassmannians}, Hermitian-Grassmannian submanifolds, 69–82, Springer Proc. Math. Stat., 203, Springer, Singapore, 2017.
%
\bibitem{LS2020} H. Lee and Y.J. Suh, {\it Cyclic parallel hypersurfaces in complex Grassmannians of rank~2}, Internat. J. Math.~{\bf 31} (2020), no. 2, 2050014, 14 pp.
%
%
\bibitem{MPS 2015} C.J.G. Machado, J.D. P\'{e}rez, and Y.J. Suh, {Commuting structure Jacobi operator for real hypersurfaces in complex two-plane Grassmannians}, Acta Math. Sin. (Engl. Ser.)~{\bf 31} (2015), no. 1, 111–122.
%
\bibitem{MUS} S. Mallick, U.C. De and Y.J. Suh, {\it Spacetimes with differential forms of energy momentum tensor}, J. Geom. Phys.~{\bf 151} (2020), 103622.
%
\bibitem{MMSS} C.A. Mantica, L.G. Molinari, Y.J. Suh and S. Shenawy, {\it Perfect-fluid, generalized Robertson-Walker space-times, and Gray's decomposion}, J. Math. Phys.~{\bf 60} (2019), no. 5, 052506, 9pp.
%
\bibitem{MUSM} C.A. Mantica, U.D. De, Y.J. Suh and L.G. Molinari, {\it Perfect fluid spacetimes with harmonic generalized curvature tensor}, Osaka J. Math.~{\bf 56} (2019), no. 1, 173-182.
%
%
\bibitem{PeSu} J.D. P{\'e}rez and Y.J. Suh, {\it The Ricci tensor of real hypersurfaces in complex two-plane Grassmannians}, J. Korean Math. Soc.~{\bf 44} (2007), 211-235.
%
%
\bibitem{REB} R. Rani, S.B. Edgar and A. Barnes, {\it Killing tensors and conformal Killing tensors from conformal Killing vectors}, Classical Quantum Gravity~{\bf 20} (2003), no. 11, 1929-1942.
%
\bibitem{Semm03} U. Semmelmann, {\it Conformal Killing forms on riemannian manifolds}, Math. Z.~{\bf 245} (2003), 503-527.
%
\bibitem{ShGh} R. Sharma and A. Ghosh, {\it Perfect fluid space-times whose energy-momentum tensor is conformal Killing}, J. Math. Phys.~{\bf 51} (2010), no. 2, 022504, 5 pp.
%
\bibitem{Suh2006} Y.J. Suh, {\it Real hypersurfaces of type B in complex two-plane Grassmannians}, Monatsh. Math.~{\bf 147} (2006), no. 4, 337–355.
%
\bibitem{Suh2012} Y.J. Suh, {\it Real hypersurfaces in complex two-plane Grassmannians with parallel Ricci tensor}, Proc. Roy. Soc. Edinburgh Sect. A~{\bf 142} (2012), no. 6, 1309–1324.
%
\bibitem{Suh2013}Y.J. Suh, {\it Real hypersurfaces in complex two-plane Grassmannians with harmonic curvature}, J. Math. Pures Appl.~{\bf 100} (2013), no. 1, 16–33.
%
\bibitem{S2013} Y.J. Suh, {\it Real hypersurfaces in complex two-plane Grassmannians with Reeb parallel Ricci tensor}. J. Geom. Phys.~{\bf 64} (2013), 1–11.
%
\bibitem{S2020} Y.J. Suh, {\it Generalized Killing Ricci tensor for real hypersurfaces in the complex two-plane Grassmannians}, to appear in J. Geom. Phys.~(2020) 103799, https://doi.org/10.1016/j.geophys.2020.103799.

\bibitem{SCP} Y.J. Suh, V. Chavant, N.A. Pundeer, {\it Pseudo-quasi-conformal curvatrue tensor and spacetimes of general relativity}, to appear in Filomat~(2021).

\end{thebibliography}
\end{document}